\documentclass[11pt]{article}

\usepackage{amsmath,amsfonts,amssymb,amsthm,amscd}
\usepackage{a4wide}
\usepackage{multirow}
\usepackage{graphicx}
\usepackage[titletoc]{appendix}

\usepackage{url}

\usepackage{hyperref}
\usepackage{breakurl}
\usepackage[table,x11names]{xcolor}
\usepackage{hhline}
\usepackage{color}
\usepackage{colortbl} 
\usepackage{tabu}     

\definecolor{gray}{rgb}{.85,.85,.85}
\definecolor{LightCyan}{rgb}{.85,1,.9}

\definecolor{seda}{rgb}{.7,.7,.7}
\definecolor{modra}{rgb}{0,0,.8}
\definecolor{piros}{rgb}{.8,0,0}
\definecolor{zluta}{rgb}{1,.9,.6}

\definecolor{modra3}{rgb}{.1,.0,.4}
\hypersetup{colorlinks=true, linkcolor=modra3, urlcolor=modra3, citecolor=modra3, pdfpagemode=UseNone, pdfstartview=} 

\if1     
\usepackage[mathlines]{lineno}
\newcommand*\patchAmsMathEnvironmentForLineno[1]{%
  \expandafter\let\csname old#1\expandafter\endcsname\csname #1\endcsname
  \expandafter\let\csname oldend#1\expandafter\endcsname\csname end#1\endcsname
  \renewenvironment{#1}%
     {\linenomath\csname old#1\endcsname}%
     {\csname oldend#1\endcsname\endlinenomath}}%
\newcommand*\patchBothAmsMathEnvironmentsForLineno[1]{%
  \patchAmsMathEnvironmentForLineno{#1}%
  \patchAmsMathEnvironmentForLineno{#1*}}%
\AtBeginDocument{%
\patchBothAmsMathEnvironmentsForLineno{equation}%
\patchBothAmsMathEnvironmentsForLineno{align}%
\patchBothAmsMathEnvironmentsForLineno{flalign}%
\patchBothAmsMathEnvironmentsForLineno{alignat}%
\patchBothAmsMathEnvironmentsForLineno{gather}%
\patchBothAmsMathEnvironmentsForLineno{multline}%
}
\linenumbers
\fi

\DeclareMathOperator{\Sym}{Sym}

\newtheorem{theorem}{Theorem}[section]   

\newtheorem{proposition}[theorem]{Proposition}
\newtheorem{observation}[theorem]{Observation}
\newtheorem{lemma}[theorem]{Lemma}

\newtheorem{fact}[theorem]{Fact}

\newtheorem{corollary}[theorem]{Corollary}
\newtheorem{problem}{Problem}

\newcommand{\R}{{\mathbb{R}}}
\newcommand{\Q}{{\mathbb{Q}}}

\renewcommand{\S}{{\mathbb{S}}}

\newcommand\eps{\varepsilon}

\newcommand{\T}{{\mathcal{T}}}

\def\myendmatrix{\end{array}\right)}
\def\mybeginmatrix{\left(\begin{array}{rrrrr}}

\def\emm{m}

\def\inst#1{$^{#1}$}

\def\marrow{{\boldmath {\marginpar[\hfill$\Rrightarrow
\Rrightarrow$]{$\Lleftarrow \Lleftarrow$}}}}
\def\jk#1{\color{modra} \ifhmode\newline\fi{\,}\hrule\vskip 1mm\hskip -2cm {\sc\small JK:}
{\marrow\sf #1 }\vskip 1mm\hrule\vskip 1mm \normalcolor}
\def\zs#1{\color{piros} \ifhmode\newline\fi{\,}\hrule\vskip 1mm\hskip -2cm {\sc\small ZS:}
{\marrow\sf #1 }\vskip 1mm\hrule\vskip 1mm \normalcolor}

\begin{document}

\title{On the nonexistence of $k$-reptile simplices in $\mathbb R^3$ and $\mathbb R^4$
\thanks{
The authors were supported by the project CE-ITI (GA\v{C}R P202/12/G061) of the Czech Science Foundation,
by the grant SVV-2015-260223 (Discrete Models and Algorithms),
by project GAUK 52410, and by ERC Advanced Research
Grant no 267165 (DISCONV). The research was partly conducted during the Special
Semester on Discrete and Computational Geometry at \'Ecole Polytechnique
F\'ederale de Lausanne, organized and supported by the CIB (Centre
Interfacultaire Bernoulli) and the SNSF (Swiss National Science Foundation).
} 
} 

\author{Jan Kyn\v cl\inst{1}${}^,$\inst{2} and Zuzana Pat\'akov\'a\inst{1}}

\date{}

\maketitle

\begin{center} 
{\footnotesize
\inst{1}
Department of Applied Mathematics and Institute for Theoretical Computer
Science, \\
Charles University, Faculty of Mathematics and Physics, \\
Malostransk\'e n\'am.~25, 118~ 00 Praha 1, Czech Republic; \\
\texttt{kyncl@kam.mff.cuni.cz, zuzka@kam.mff.cuni.cz}
\\\ 
\inst{2}
Alfr\'ed R\'enyi Institute of Mathematics, Re\'altanoda u. 13-15, Budapest 1053, Hungary
}
\end{center}  


\begin{abstract}
A $d$-dimensional~simplex $S$ is called a \emph{$k$-reptile} (or a
\emph{$k$-reptile simplex}) if it can be tiled by $k$ simplices
with disjoint interiors that are all mutually congruent and similar to $S$. 
For $d=2$, triangular $k$-reptiles exist for all $k$ of the form $a^2, 3a^2$ or $a^2 + b^2$
and they have been completely 
characterized by Snover, Waiveris, and Williams. On the other hand, the only
$k$-reptile simplices that are known for $d \ge 3$, have $k = m^d$, where 
$m$ is a positive integer. 
We substantially simplify the proof by Matou\v{s}ek and the second author that for $d=3$, $k$-reptile tetrahedra can exist only for $k=m^3$.
We then prove a weaker analogue of this result for $d=4$ by showing that four-dimensional $k$-reptile
simplices can exist only for $k=m^2$.
\end{abstract}


\section{Introduction}

A {\em tiling\/} of a closed set $X$ in $\R^d$ (or in the unit sphere $S^d$) is a locally finite decomposition $X=\bigcup_{i\in I}X_i$ into closed sets with nonempty and pairwise disjoint interiors. The sets $X_i$ are called {\em tiles\/}. If $X$ has a tiling where all the tiles are congruent to a set $T$, we say that $T$ {\em tiles\/} $X$, or, that $X$ {\em can be tiled with (\/$|I|$ copies of)} $T$. We emphasize that congruence includes mirror symmetries.

A closed set $X\subset\R^d$ with nonempty interior is called a {\em
$k$-reptile\/} (or a {\em $k$-reptile set}) if $X$ can be tiled with $k$ mutually congruent copies of a set similar to $X$.

It is easy to see that whenever $S$ is
a  $d$-dimensional $k$-reptile set, then $S$ is {\em space-filling}, that is, the space $\R^d$
can be tiled with $S$: indeed, using
the tiling of $S$ by its smaller copies 
as a pattern, one can inductively tile larger and larger similar
copies of~$S$. On the other hand, it is a simple exercise to find space-filling polytopes or polygons
that are not $k$-reptiles for any $k\ge 2$.

Clearly, every triangle tiles $\R^2$. Moreover, every triangle $T$ is a $k$-reptile for $k=m^2$, since $T$ can be tiled in a regular way with $m^2$ congruent tiles, each positively or negatively homothetic to $T$. See Snover et al.~\cite{snover} for an illustration.

In this paper we study the existence of $k$-reptile {\em simplices\/} in $\R^d$, especially for $d=3$ and $d=4$.

\paragraph{Space-filling simplices.}
The question of characterizing the tetrahedra that tile $\R^3$ 
is still open and apparently rather difficult. The first systematic study
of space-filling tetrahedra was made by Sommerville. 
Sommerville~\cite{sommerville-2} discovered a list of exactly four tilings (up
to isometry and rescaling), but he assumed that all tiles are {\em properly
congruent\/} (that is, congruent by an orientation-preserving 
isometry) and meet face-to-face. Edmonds~\cite{edmonds} noticed a gap in Sommerville's proof and by completing the analysis, he confirmed that Sommerville's classification of proper, face-to-face tilings is complete. Baumgartner~\cite{baumgartner} found three of Sommerville's tetrahedra and one new tetrahedron that admits a non-proper face-to-face tiling (and also a proper non face-to-face tiling~\cite{goldberg}).
Goldberg~\cite{goldberg} described three families of proper (generally not face-to-face) tilings, obtained by partitioning a triangular prism. In fact, Goldberg's first family was found by Sommerville~\cite{sommerville-2} before, but he selected only special cases with a certain symmetry. Goldberg's first family also coincides with the family of simplices found by Hill~\cite{hill}, whose aim was to classify {\em rectifiable} simplices, that is, simplices that can be cut by straight cuts into finitely many pieces and rearranged to form a cube. The simplices in Goldberg's second and third families are obtained from the simplices in the first family by splitting into two congruent halves. According to Senechal's survey~\cite{senechal}, no other space-filling tetrahedra 
are known. 

For $d \ge 3$, Debrunner~\cite{debrunner} constructed $\lfloor d/2 \rfloor + 2$ one-parameter families and several special types of $d$-dimensional simplices that tile $\R^d$. Smith~\cite{WSmith-pythag} generalized Goldberg's construction and using Debrunner's ideas, he obtained $(\lfloor d/2 \rfloor + 2)\phi(d)/2$ one-parameter families of space-filling $d$-dimensional simplices; here $\phi(d)$ is the Euler's totient function. 
It is not known whether for some $d  \ge 3$ there is a space-filling simplex with all dihedral angles acute
or a two-parameter family of space-filling simplices~\cite{WSmith-pythag}.

\paragraph{Hilbert's problems.}
Two Hilbert's problems are related to tilings of the Euclidean space.
The second part of Hilbert's 18th problem asks whether there exists a polyhedron that tiles the $3$-dimensional Euclidean space but does not admit an isohedral (tile-transitive) tiling. The first such tile in three dimensions
was found by Reinhardt~\cite{reinhardt}. Later Heesch~\cite{heesch} found a planar anisohedral nonconvex polygon and  Kershner~\cite{kershner} found an anisohedral convex pentagon. Hilbert's 18th problem was discussed in detail by Milnor~\cite{milnor}. See also the survey by Gr\"unbaum and Shepard~\cite{grunbaum_shephard} for a discussion of this problem and related questions. While iterated tilings of the space using tilings of some $k$-reptiles as a pattern may be highly irregular, it is an interesting question whether there is an anisohedral $k$-reptile polytope or polygon. Vince~\cite[Question 2]{vince95} asked whether there is a $k$-reptile that admits no periodic tiling. 

The third Hilbert's problem asks whether two tetrahedra with equal bases and altitudes are {\em equidecomposable}, that is, whether one can cut one tetrahedron into finitely many polytopes and reassemble them to form the second tetrahedron. A positive answer would provide an elementary proof of the formula for the volume of the tetrahedron. However, Dehn~\cite{Dehn02_rauminhalt} answered the question in the negative, by introducing an algebraic invariant for equidecomposability. See~\cite{Do06_scissors} for an elementary exposition or~\cite[Chapter 9]{aigner_ziegler},~\cite{benko} for alternative proofs.
Debrunner~\cite{Deb80_fill_cube_R3R4} proved that every polytope that tiles $\mathbb{R}^d$ has its {\em codimension 2 Dehn's invariant\/} equal to zero. Lagarias and Moews~\cite{LM95_priority,LM95_fill_cube_R3R4} showed that, more generally, every polytope that tiles $\mathbb{R}^d$ has its {\em classical total Euclidean Dehn's invariant\/} equal to zero. In particular, these properties are necessary for every $k$-reptile simplex. For $d=3$ and $d=4$, the results of Sydler~\cite{Sydler65_cube_dehn} (see also~\cite{Jes68_R3_sydler}) and Jessen~\cite{Jes72_R4_cube_dehn} imply that every polytope that tiles $\mathbb{R}^d$ is equidecomposable with a cube~\cite{Deb80_fill_cube_R3R4,LM95_priority,LM95_fill_cube_R3R4}. 

\paragraph{Reptiles and other animals.}
Motivated by classical puzzles that require splitting a given figure into a given number of congruent {\em replicas\/} of the original figure, Langford~\cite{Langford40} initiated a systematic study of planar $k$-reptiles. Golomb~\cite{golomb64} introduced the term {\em replicating figure of order $k$}, shortly {\em a rep-$k$}, and described several more examples, including disconnected or totally disconnected fractal tiles. See also Gardner's~\cite{gardner91_hanging} short survey.  
Extending the theory of self-similar sets and fractals, Bandt~\cite{bandt91} described a general construction of infinitely many self-similar $k$-reptiles, including several species of dragons, which are examples of {\em disk-like\/} (that is, homeomorphic to a disk) reptiles. Gelbrich~\cite{gelbrich93} proved that for every $k$, there are only finitely many planar disk-like crystallographic (isohedral) $k$-reptiles.
See Gelbrich and Giesche~\cite{GG98_salamanders} for illustrations of several such $7$-reptiles, such as sea horses or salamanders.
Vince~\cite{vince95} studied lattice reptiles and their connection with generalized number systems.

\paragraph{$k$-reptile simplices.}
In recent years the subject of tilings has received a certain impulse from
computer graphics and other 
computer applications. In fact, our original motivation for studying simplices that
are $k$-reptiles comes 
from a problem of probabilistic marking of Internet packets for IP
traceback~\cite{adler,adler_et_al}. See~\cite{matousek} for a brief summary of
the ideas of this method.
For this application, it would be interesting to find a $d$-dimensional 
simplex that is a $k$-reptile with $k$ as small as possible.

For dimension 2 there are several possible types of $k$-reptile triangles, and
they
have been completely classified by Snover et al.~\cite{snover}. 
In particular, $k$-reptile triangles exist for all $k$ 
of the form $a^2+b^2$, $a^2$ or $3a^2$ for arbitrary integers $a,b$. 
In contrast, for $d\ge 3$, reptile simplices seem to be much more rare.
The only known constructions of higher-dimensional $k$-reptile simplices
have $k=m^d$. The best known examples are the {\em Hill simplices\/}     
(or the {\em Hadwiger--Hill simplices\/})~\cite{debrunner,hadwiger,hill}.
A $d$-dimensional Hill simplex is the convex hull of vectors
$0$, $b_1$, $b_1+b_2$, \dots, $b_1+b_2+ \cdots +b_d$,
where $b_1,b_2,\dots,b_d$ are vectors of equal length
such that the angle between every two of them is the same
and lies in the interval $(0,\frac{\pi}{2}+\arcsin{\frac{1}{d-1}})$.

Hertel~\cite{hertel} proved that a $3$-dimensional
simplex is an $m^3$-reptile using a ``standard'' way of 
dissection (which we will not define here) if and only if it is a Hill simplex.
He conjectured that Hill simplices are the only 3-dimensional
reptile simplices. Herman Haverkort recently pointed us to an example of a
$k$-reptile
tetrahedron by Liu and Joe~\cite{liujoe94} which is not Hill, and thus contradicts Hertel's conjecture. In fact, except for the one-parameter family of Hill tetrahedra, two other space-filling tetrahedra described by Sommerville~\cite{sommerville-2} and Goldberg~\cite{goldberg} are also $k$-reptiles for every $k=m^3$. Both these tetrahedra tile the right-angled Hill tetrahedron, and their tilings are based on the barycentric subdivision of the cube. 
Maehara~\cite{maehara13} described a generalized construction of $d$ distinct $k$-reptile simplices in $\mathbb{R}^d$ for $k=2^d$. It is easy to see that the lattice tiling of $\mathbb{R}^d$ by barycentrically subdivided unit cubes can be obtained by cutting the space with hyperplanes $x_i=n/2$, $x_i+x_j=n$, $x_i-x_j=n$, for every $i,j\in[d], i\neq j$ and $n\in \mathbb{Z}$. Each tile in this tiling is congruent to the right-angled Hill simplex $H^0_d$ defined as the convex hull of points $(0,0,\dots,0)$, $(1/2,0,\dots,0)$, \dots, $(1/2,1/2,\dots,1/2)$.
For every $m$, this tiling contains a tiling of an $m$ times scaled copy of $H^0_d$. Similarly, by removing the hyperplanes $x_i=(2n+1)/2$ from the cutting, we obtain a tiling of $\mathbb{R}^d$ with tiles that are made of two copies of $H^0_d$; more precisely, each tile is congruent to the simplex $H^1_d$ defined as the convex hull of the points $(0,0,\dots,0)$, $(1,0,\dots,0)$, $(1/2,1/2,0,\dots,0)$, \dots, $(1/2,1/2,\dots,1/2)$. Again, for every $m$, this tiling contains a tiling of an $m$ times scaled copy of $H^1_d$.

Let $H^2_d$ be the convex hull of the points $(0,0,\dots,0)$, $(1,0,\dots,0)$, $(1,\allowbreak1,\allowbreak0,\dots,0)$, $(1/2,\allowbreak 1/2,\allowbreak 1/2,\allowbreak 0,\dots,0), \dots, (1/2,\allowbreak 1/2, \dots, 1/2)$. The simplex $H^2_d$ can be tiled with two copies of $H^1_d$ or four copies of $H^0_d$. Let $m$ be a fixed positive integer and let $m\cdot H^2_d$ be the $m$ times scaled copy of $H^2_d$ obtained from $H^2_d$ by multiplying all the coordinates of all its points by $m$. The tiling described in the previous paragraph provides a tiling of $m\cdot H^2_d$ by $2m^d$ copies of $H^1_d$. To obtain a tiling of $m\cdot H^2_d$ by copies of $H^2_d$, it is enough to join the copies of $H^1_d$ into $m^d$ disjoint pairs so that each pair forms a copy of $H^2_d$. 
Every tile $H$ in the tiling of $\mathbb{R}^d$ by copies of $H^1_d$ can be represented by the center $\mathbf{z}=(n_1+1/2, n_2+1/2, \dots, n_d+1/2)$ of a unit cube it contains and by a signed permutation $(\varepsilon_1i_1, \varepsilon_2i_2, \dots, \varepsilon_{d-1}i_{d-1})$ where $\varepsilon_i\in\{-1,1\}$, $i_j\in [d]$ and $i_j\neq i_k$ if $j\neq k$. Let $\{i_d\}=\{1,2,\dots,d\}\setminus \{i_1,i_2,\dots,i_{d-1}\}$ and let $\mathbf{e}_i$ be the $i$th unit vector of the canonical basis. The tile $H$ is then the convex hull of the points 
\begin{align*}
& \mathbf{z},\\ 
& \mathbf{z}+ \varepsilon_1\mathbf{e}_{i_1},\ \dots \ , \\
& \mathbf{z}+ \varepsilon_1\mathbf{e}_{i_1} + \varepsilon_2\mathbf{e}_{i_2} + \dots + \varepsilon_{d-2}\mathbf{e}_{i_{d-2}},\\
& \mathbf{z}+ \varepsilon_1\mathbf{e}_{i_1} + \varepsilon_2\mathbf{e}_{i_2} + \dots + \varepsilon_{d-2}\mathbf{e}_{i_{d-2}}+\varepsilon_{d-1}\mathbf{e}_{i_{d-1}} + \mathbf{e}_{i_{d}},\\
& \mathbf{z}+ \varepsilon_1\mathbf {e}_{i_1} + \varepsilon_2\mathbf{e}_{i_2} + \dots + \varepsilon_{d-2}\mathbf{e}_{i_{d-2}}+\varepsilon_{d-1}\mathbf{e}_{i_{d-1}} - \mathbf{e}_{i_{d}}.
\end{align*}
We say that a tile $H$ is \emph{compatible} with a tile $H'$ if their union is a simplex congruent to $H^2_n$.
A tile $H$ represented by $\mathbf{z}$ and $(\varepsilon_1i_1, \varepsilon_2i_2, \dots, \varepsilon_{d-1}i_{d-1})$ and a tile $H'$ represented by $\mathbf{z'}$ and $(\varepsilon'_1i'_1, \varepsilon'_2i'_2, \dots, \varepsilon'_{d-1}i'_{d-1})$ are compatible if and only if $\mathbf{z}=\mathbf{z'}$ and $(\varepsilon_1i_1, \varepsilon_2i_2, \dots, \varepsilon_{d-2}i_{d-2})=(\varepsilon'_1i'_1, \varepsilon'_2i'_2, \dots, \varepsilon'_{d-2}i'_{d-2})$. In particular, each tile $H$ is compatible with two other neighboring tiles, and every component in the corresponding \emph{compatibility graph} $\mathcal{G}$ is a four-cycle. 
Let $H$ be a tile represented by $\mathbf{z}$ and $(\varepsilon_1i_1, \varepsilon_2i_2, \dots, \varepsilon_{d-1}i_{d-1})$. The four tiles $H,H',H',H'''$ forming a component of $\mathcal{G}$ containing $H$ are separated by hyperplanes orthogonal to the vectors $x_{i_{d-1}}+x_{i_{d}}$ and $x_{i_{d-1}}-x_{i_{d}}$. Since at most one of these hyperplanes determines a facet of $m\cdot H^2_d$, the simplex $m\cdot H^2_d$ contains an even number of the tiles $H,H',H',H'''$, which can be matched into zero, one or two compatible pairs. 

Let $H^i_d$ be the $d$-dimensional simplex $\sigma_i$ described by Maehara~\cite{maehara13}. These simplices satisfy $\sigma_i=H^i_d$ for $i\in \{0,1,2\}$, $\sigma_d=2\cdot H^0_d$ and in general, for each $i\in [d]$, the simplex $H^{i}_d$ can be tiled with two copies of $H^{i-1}_d$. Since for every positive integer $m$ the simplex $H^0_d$ is $m^d$-reptile, each of the simplices $H^{i}_d$ is $(2m)^d$-reptile. Except for $i\in\{0,1,2\}$, we do not know whether $H^{i}_d$ is $m^d$-reptile for odd $m\ge 3$.

\begin{problem}
Let $d$ and $i$ be positive integers satisfying $3\le i\le d-1$ and let $m\ge 3$ be an odd integer. Is it true that the simplex $H^i_d$ is $m^d$-reptile?
\end{problem}

Matou\v{s}ek~\cite{matousek} showed that
there are no $2$-reptile simplices of dimension $3$ or larger.
For dimension $d=3$, Matou\v{s}ek and the second author~\cite{MS} proved the
following theorem.

\begin{theorem}\label{veta_3dim}{\rm \cite{MS}}
In $\R^3$, $k$-reptile simplices (tetrahedra) exist only for $k$ of the form
$m^3$ where $m$ is a positive integer.
\end{theorem}

We give a new simple proof of Theorem~\ref{veta_3dim} in Section~\ref{section_dukaz3dim}.

Matou\v{s}ek and the second author~\cite{MS} conjectured that for $d\ge 3$, a $d$-dimensional $k$-reptile simplex can exist only for $k$ of the form $m^d$ for some positive integer $m$.
We prove a weaker version of this conjecture for four-dimensional simplices.

\begin{theorem}\label{veta_hlavni} 
Four-dimensional $k$-reptile simplices can
exist only for $k$ of the form $m^2$, where $m$ is a positive integer.
\end{theorem}

Four-dimensional Hill simplices are examples of $k$-reptile simplices for $k=m^4$. 
However, the following question remains open.

\begin{problem}
Is there a four-dimensional $m^2$-reptile simplex for $m$ non-square?
\end{problem}

\paragraph{New ingredients.}
Debrunner's lemma~\cite{debrunner} connects
the symmetries of a $d$-simplex with the symmetries of its Coxeter diagram
(which represents the ``arrangement'' of the dihedral angles), and is an important tool in our analysis. This lemma allows us to substantially simplify the proof of
Theorem~\ref{veta_3dim} 
and enables us to step one dimension up and prove Theorem~\ref{veta_hlavni},
which seemed unmanageable before.

In the proof of Theorem~\ref{veta_hlavni} we encounter the problem of tiling spherical triangles by congruent
triangular tiles, which might be of independent interest. 
A related question, a
classification of edge-to-edge tilings of the sphere by congruent triangles, has been
completely solved by Agaoka and Ueno~\cite{agaoka_ueno}.


\section{Basic notions and facts about simplices and group actions}

\subsection{Angles in simplices and Coxeter diagrams}

Given a $d$-dimensional simplex $S$ with vertices $v_1, v_2, \dots, v_{d+1}$, let $F_i$
be the facet opposite to $v_i$. 
If $\alpha_{i,j}$ is the angle between the normals of $F_i$ and $F_j$ pointing outward, then the \emph{dihedral angle} $\beta_{i,j}$ is defined as $\pi-\alpha_{i,j}$.
By an \emph{internal angle $\varphi$ at the point $x$ of $S$}, where $x$ is on the boundary of $S$, we mean the set $\S^{d-1}(x,\eps) \cap S$, where $\S^{d-1}(x,\eps)$ denotes
the $(d-1)$-dimensional sphere with radius $\eps$ centered at $x$, and $\eps > 0$ is small enough so that $\S^{d-1}(x,\eps)$ does not meet the facets
not containing $x$.
An \emph{edge-angle} of $S$ is the internal $(d-1)$-dimensional angle at an interior point of an edge of $S$ 
and can be represented by a $(d-2)$-dimensional spherical simplex.
Indeed, select an interior
point $x$ of the edge $e$ and consider the hyperplane $h$ orthogonal to $e$ and containing $x$. 
The edge-angle incident to $e$ can be represented as the intersection $h \cap S \cap \S^{d-1}(x,\eps)$.
This intersection is clearly $(d-2)$-dimensional and forms a spherical simplex.

From now on we normalize all edge-angles, that is, we consider them as subsets of the $(d-2)$-dimensional unit sphere.

The \emph{Coxeter diagram} of $S$ is a graph $c(S)$ with labeled edges 
such that the vertices of $c(S)$ represent the facets of $S$ and for every pair
of facets $F_i$ and $F_j$, there is an edge $e_{i,j}$ labeled 
by the dihedral angle $\beta_{i,j}$. 
Note that our labeling differs from
the traditional one, where the edge corresponding to a dihedral angle $\pi / p$
is labeled by $p$ and the label $3$ is omitted. Debrunner~\cite{debrunner}
labels the edge corresponding to a dihedral angle $\beta_{i,j}$ by $\cos
\beta_{i,j}$.

\begin{observation}
The edge-angles of a four-dimensional simplex $S$ can be represented by
spherical triangles, whose angles are dihedral angles in $S$. Therefore, an
edge-angle in $S$ represented by a spherical triangle with angles $\alpha,
\beta, \gamma$ corresponds to a triangle in the Coxeter diagram with edges
labeled by $\alpha, \beta, \gamma$.
\qed
\end{observation}

Debrunner~\cite[Lemma 1]{debrunner} proved the following important lemma. Here the {\em symmetries of S\/} are Euclidean isometries, and the {\em symmetries of c(S)\/} are graph automorphisms preserving the labels of edges.

\begin{lemma}[{\bf Debrunner's lemma}~\cite{debrunner}]\label{lemma_debrunner}
Let $S$ be a $d$-dimensional simplex. The symmetries of $S$ are in one-to-one
correspondence with the symmetries 
of its Coxeter diagram $c(S)$, in the following sense: each symmetry $\varphi$ of
$S$ induces a symmetry $\Phi$ of $c(S)$ so that $\varphi(v_i)=v_j
\Leftrightarrow \Phi(F_i)=F_j$, and vice versa. 
\end{lemma}

\subsection{Existence of simplices with given dihedral angles} 

Fiedler~\cite{fiedler} proved the following elegant property of ${\binom{d+1}{2}}$-tuples of dihedral angles. A proof in English can be found in~\cite{MS}. 

\begin{theorem}[{\bf Fiedler's theorem}~\cite{fiedler}]\label{theorem_fiedler}
Let $\beta_{i,j}$, $i,j=1,2,\dots,d+1$, be the dihedral angles of some
$d$-dimensional simplex, let 
$\beta_{i,i} = \pi$ for convenience, and let $A$ be the $(d+1)\times(d+1)$
matrix with
 $a_{i,j}:=\cos\beta_{i,j}$ for all $i,j$. Then $A$ is negative semidefinite of
rank $d$, and the 
($1$-dimensional) kernel of $A$ is generated by a vector $z\in\R^{d+1}$ with all
components strictly positive.
\end{theorem}

In our proof of Theorem~\ref{veta_hlavni} we use only the fact that the matrix $A$ defined in Theorem~\ref{theorem_fiedler} is singular; indeed, it is a $(d+1)\times(d+1)$ matrix of rank $d$.

\subsection{Group actions}

An \emph{action} $\varphi$ of a group $G$ on a set $M$ is a homomorphism from $G$ to the symmetric group $\Sym(M)$ of $M$, where the symmetric group $\Sym(M)$ is the group of all permutations of $M$.
We say that an action $\varphi$ of $G$ on $M$ is \emph{faithful} if its kernel is trivial. In other words,
$\varphi$ is faithful if for every $g\neq 1$ there exists an element $m\in M$ with $\varphi(g)(m)\neq m$. 
It is usual to omit $\varphi$ and write just $gm$ instead of $\varphi(g)(m)$. 

The set $Gm:=\{gm\colon g\in G\}$ is called the \emph{orbit} of the element $m$ under the action of $G$. It is obvious that the set of orbits forms a partition of $M$. 
The following well-known lemma counts the number of orbits in the partition. 

\begin{lemma}[{\bf Burnside's lemma}~\cite{burnside}]
Let $M$ be a finite set  and $G$ a finite group acting on $M$ via $m \mapsto gm$.
 Let $X_g $ be the number of elements of $M$ fixed by $g$, that is, those satisfying the identity $gm=m$.
 Then the action  of $G$ on $M$ has exactly $\frac{1}{|G|}\sum_{g\in G}X_g$ orbits.
 \end{lemma}
 
 We will need the following lemma:
 
\begin{lemma}\label{l_pocet_orbit}
 Let $M$ be a finite set  and $G$ a finite group acting on $M$ nontrivially and faithfully via $m \mapsto gm$. 
 Then $G$ also acts on the (unordered) pairs $\{m,n\} \in \binom{M}{2}$ via $g\{m,n\}=\{gm,gn\}$ and
 the action of $G$ on pairs has at most $\binom{|M|}{2} - |M| + 2$ orbits.
 Moreover, the bound is  tight and it is achieved if 
 the image of $G$ under the action is generated by a single transposition.
\end{lemma}

\begin{proof}
 Let $o_1$ be the number of orbits of the action of $G$ on $M$. Let $o_2$ be the number of orbits of the induced action of $G$ on $\binom{M}{2}$. Since the action on $M$ is nontrivial, we have $o_1 \le |M|-1$.
 
 Let $X_g $ be the number of elements of $M$ fixed by $g$. 
 We show that the number of elements of $\binom{M}{2}$ fixed by $g$ is $\binom{X_g}{2} + \frac{1}{2}(X_{g^2}-X_g)=\frac{1}{2}(X^2_{g}+X_{g^2})-X_g$.
 Indeed, there are two possibilities for stabilizing the pair $\{m,n\}$:
 \begin{enumerate}
  \item $gm=m$ and $gn=n$; this gives $\binom{X_g}{2}$ fixed elements of $\binom{M}{2}$.
  \item $gm=n$ and $gn=m$; this can be rewritten as $ggn=n$ and $gn\neq n$. 
  Thus in this case we have 
$\frac{1}{2}(X_{g^2}-X_g)$ fixed elements of $\binom{M}{2}$.
 \end{enumerate}
By Burnside's lemma, we have
\begin{equation}
 o_2 = \frac{1}{|G|}\sum_{g \in G}\left(\frac{1}{2}(X^2_{g}+X_{g^2})-X_g\right).\label{eq:o_2}
\end{equation}
In order to bound \eqref{eq:o_2}, we need to bound $\sum X_{g}^2$ in terms of $\sum X_g$:
\begin{equation}
 \sum_{g \in G}X^2_g \le (|M|-2)\sum_{g \in G}X_g + 2|M|. \label{eq:2}
\end{equation}
Indeed, the action is faithful and nontrivial, hence $X_{1}=|M|$ and $X_{g} \le |M|-2$ otherwise. Using $\sum_{g \neq 1}X_g^2 \le (|M|-2)\sum_{g \neq 1}X_g$, the bound \eqref{eq:2} follows.

Plugging \eqref{eq:2} into \eqref{eq:o_2} and using $X_{g^2} \le |M|$ we have
\begin{align*}
 o_2 & =  \frac{1}{2|G|}\sum_{g \in G} X_g^2 + \frac{1}{2|G|}\sum_{g \in G} X_{g^2} - \frac{1}{|G|}\sum_{g \in G} X_g \\
 &\le  \frac{|M|}{2|G|}\sum_{g \in G} X_g - \frac{2}{|G|}\sum_{g \in G} X_g + \frac{|M|}{|G|} + \frac{|M|}{2}.
\end{align*}
By Burnside's lemma for the action on $M$, we have $\frac{1}{|G|}\sum X_g=o_1\le|M|-1$. Since$|G| \ge 2$, we get the desired bound:
\[
 o_2 \le \frac{(|M|-4)(|M|-1)}{2} + |M| = \binom{|M|}{2} - |M| + 2.
\]

 It remains to show the last part of the statement. But this is clear, since a single transposition swaps $|M|-2$ pairs of edges. 
\end{proof}

\section{A simple proof of Theorem~\ref{veta_3dim}}\label{section_dukaz3dim} 
We proceed as in the original proof~\cite{MS}, but instead of using the theory of scissors
congruence, Jahnel's 
theorem about values of rational angles and Fiedler's theorem, we only use
Debrunner's lemma (Lemma~\ref{lemma_debrunner}). 

Assume for contradiction that $S$ is a $k$-reptile tetrahedron where $k$ is not
a third power of 
a positive integer. A dihedral angle $\alpha$ is called \emph{indivisible} if it
cannot be written 
as a linear combination of other dihedral angles in $S$ with nonnegative integer
coefficients. 

The following lemmas are proved in~\cite{MS}.

\begin{lemma}\label{lemma_aspon3}{\rm\cite[Lemma 3.1]{MS}}
If $\alpha$ is an indivisible dihedral angle in $S$, then the edges of $S$ with
dihedral angle $\alpha$ have at least three different lengths.
\end{lemma}

Lemma~\ref{lemma_aspon3} is analogous to Lemma~\ref{lemma_at-least-4-edges}, which
we prove in the next section. 

\begin{lemma}\label{lemma_onlytwo}{\rm\cite[Lemma 3.3]{MS}}
One of the following two possibilities occur: 
\begin{enumerate} 

\item[\rm(i)] All the dihedral angles of $S$ are integer multiples of the
minimal dihedral angle 
$\alpha$, which has the form $\frac\pi n$ for an integer $n\ge 3$. 

\item[\rm(ii)] There are exactly two distinct dihedral angles $\beta_1$ and
$\beta_2$, each
of them occurring three times in $S$. 
\end{enumerate}
\end{lemma}

First we exclude case (ii) of Lemma~\ref{lemma_onlytwo}. If $S$ has two distinct
dihedral angles $\beta_1 \neq \beta_2$, each occurring at three edges,
then they can be placed in $S$ in two essentially different ways; see Figure~\ref{fig_tri-pa}. 
In both cases, for each $i \in \{1,2\}$, the Coxeter diagram of $S$ has at least
one nontrivial symmetry swapping two distinct edges with label $\beta_i$. By
Debrunner's lemma, the corresponding symmetry of $S$ swaps two distinct edges
with dihedral angle $\beta_i$, which thus have the same length. But then the
edges with dihedral angle $\beta_i$ have at most two different lengths and this
contradicts Lemma~\ref{lemma_aspon3}, since the smaller of the two angles $\beta_1,\beta_2$ is indivisible.

\begin{figure}[htb]
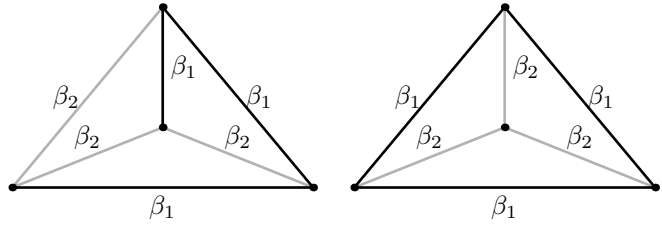

 \begin{center}
  \begin{tabular}{cc}
 \includegraphics{obrazky/coxeter-3d.0}  &
\includegraphics{obrazky/coxeter-3d.1}  \\
  \end{tabular}
\caption{Two possible configurations of two dihedral angles.}
\label{fig_tri-pa}
 \end{center}
\end{figure}

Now we exclude case (i) of Lemma~\ref{lemma_onlytwo}. Call the edges of $S$ (and of
$c(S)$) with dihedral angle $\alpha$ the {\em $\alpha$-edges}. Since there are
at least three $\alpha$-edges in $S$, there is a vertex $v$ of $S$ where two
$\alpha$-edges meet. Let $\beta$ be the dihedral angle of the third edge
incident to $v$ (possibly $\beta$ can be equal to $\alpha$). 
In the proof of Lemma 3.5 in~\cite{MS} it was shown that $\beta=\pi-\alpha$. 
For the Coxeter diagram of $S$, this implies that whenever two $\alpha$-edges meet in $c(S)$, then the label of the edge forming a triangle with the two $\alpha$-edges is $\beta$.

Now we distinguish several cases depending on the subgraph $H_{\alpha}$ of
$c(S)$ formed by the $\alpha$-edges. 

\begin{itemize}
\item $H_{\alpha}$ contains three edges incident to a common vertex.
Then all the other edges must be labeled with
$\beta$ and thus we get a configuration like in Figure~\ref{fig_tri-pa}, right, which
we excluded earlier. 

\item $H_{\alpha}$ contains a triangle.
Then $\beta=\alpha$, and thus $\alpha=\frac\pi2$,
which contradicts the condition $n\ge 3$ from Lemma~\ref{lemma_onlytwo}(i).
 
\item $H_{\alpha}$ contains a path of length $3$.
Then two other edges have label $\beta$
and the remaining edge has some label $\gamma$ (possibly $\gamma$ can be equal to $\alpha$). See Figure~\ref{fig_four}, left. 
The symmetric group of the resulting Coxeter diagram always contains an involution swapping two $\alpha$-edges. 
 Unless $\gamma = \alpha$, there are, by Debrunner's lemma, only two $\alpha$-edge lengths; a contradiction with Lemma~\ref{lemma_aspon3}.
For $\gamma = \alpha$ the Coxeter diagram has a dihedral symmetry
group, $D_4$, acting transitively on the $\alpha$-edges; see Figure~\ref{fig_four}, right.  This again contradicts
Lemma~\ref{lemma_aspon3}, since by Debrunner's lemma, all
the $\alpha$-edges have the same length.

\end{itemize}

\begin{figure}[htb]
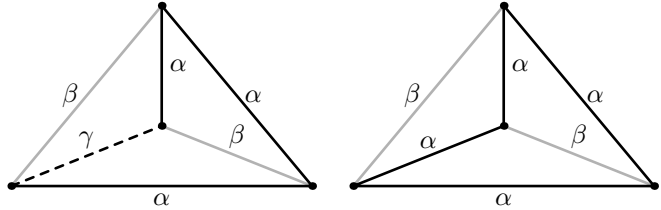

 \begin{center}
 \begin{tabular}{cc}
 \includegraphics{obrazky/coxeter-3d.2}  &
\includegraphics{obrazky/coxeter-3d.3}  \\
  \end{tabular}
\caption{The $\alpha$-edges form a path (left) or a four-cycle (right) in
$c(S)$.}
\label{fig_four}
\end{center}
\end{figure}

We obtained a contradiction in each of the cases, hence the proof of
Theorem~\ref{veta_3dim} is finished.


\section{The proof of Theorem~\ref{veta_hlavni}}\label{s:main_pf}

The method of the proof is similar to the three-dimensional case~\cite{MS}. 

Assume for contradiction that $S$ is a four-dimensional $k$-reptile simplex
where $k$ is not a square of a positive integer.
Let $S_1,S_2,\dots,S_k$ be mutually congruent simplices similar to $S$ that form a tiling of $S$. Then each $S_i$ has volume $k$-times smaller than $S$, and thus $S_i$ is
scaled by the ratio $\rho:=k^{-1/4}$ compared to $S$. For $k$ non-square, $\rho$
is an irrational number of algebraic degree 4 over $\Q$.

Similarly to~\cite{MS} we define an {\em indivisible\/} edge-angle (spherical
triangle) as a spherical triangle that cannot be tiled with smaller spherical
triangles representing the other edge-angles of $S$ or their mirror images.
Clearly, the edge-angle with the smallest spherical area is indivisible. In this paper 
we always consider a spherical triangle and its mirror image as the
same spherical triangle.

\begin{lemma}\label{lemma_at-least-4-edges}
If $\T_0$ is an indivisible edge-angle in $S$, then the edges of $S$ with
edge-angle $\T_0$ have at least four different lengths (and in particular, there
are at least four such edges). 
\end{lemma}

\begin{proof}
The proof is basically the same as for indivisible dihedral angles in
tetrahedra~\cite[Lemma 3.1]{MS}.
Let $e$ be an edge with edge-angle $\T_0$. Every point of $e$ belongs to an edge
of some of the smaller simplices $S_i$. Since $\T_0$ is indivisible, we get that
$e$ is tiled by edges of the simplices $S_i$ and each of these edges has
edge-angle $\T_0$.

Assume for contradiction that there are at most three edges with edge-angle
$\T_0$, with lengths $x_1, x_2, x_3$.
Then the edge of length $x_1$ is tiled by edges with lengths $\rho x_1$, $\rho
x_2$ and $\rho x_3$, and similarly for the edges of lengths $x_2$ and $x_3$. In
other words, there are nonnegative integers $n_{ij}$, $i,j=1,2,3$, such that

\begin{equation*}\label{thesys}
\begin{array}{rcl}
n_{11}\rho x_1 + n_{12} \rho x_2 + n_{13} \rho x_3 &=& x_1,\\
n_{21}\rho x_1 + n_{22} \rho x_2 + n_{23} \rho x_3 &=& x_2, \\
n_{21}\rho x_1 + n_{22} \rho x_2 + n_{33} \rho x_3 &=& x_3. \\
\end{array}
\end{equation*}

This can be rewritten as $\rho \mathbf A \mathbf x=\mathbf x$, where 
$\mathbf x= (x_1,x_2,x_3)^{\mathsf{T}}$ 
and $\mathbf A$ is a $3\times 3$ matrix with integer coefficients. 
Since $\mathbf x$ is nonzero, we immediately see that $1/\rho$ is an eigenvalue of $\mathbf A$.
Since the characteristic polynomial of $\mathbf A$ has degree $3$, we get a contradiction with $1/\rho$ (and hence also $\rho$) having algebraic degree $4$.
\end{proof}

Since $S$ has $10$ edges, Lemma~\ref{lemma_at-least-4-edges} implies that there are
at most two indivisible edge-angles. 

The strategy of the proof is now the following. First we exclude the case of two
indivisible edge-angles, using only elementary combinatorial arguments,
Debrunner's lemma and Lemma \ref{l_pocet_orbit}. Then we consider the case of one indivisible edge-angle. Here
we need more involved arguments: we study tilings of spherical triangles with
copies of a single spherical triangle and use various observations from
spherical geometry. We also use Fiedler's theorem (Theorem~\ref{theorem_fiedler}) to solve several cases. 


\subsection{Two indivisible edge-angles}

First, we prove an elementary observation about symmetries of the simplex $S$ and
its Coxeter diagram.

\begin{lemma}\label{lemma_7_orbits}
If $c(S)$ has a nontrivial symmetry, then the edges of $S$ have at most seven
orbits under the action of the symmetry group of $S$.
\end{lemma}

\begin{proof}
Let $M$ be the set of vertices of $S$.
By Debrunner's lemma, the symmetry groups of $S$ and $c(S)$ are isomorphic. In
particular, $S$ has a nontrivial symmetry group $\Phi\subseteq \Sym(M)$ acting faithfully on $M$. By Lemma~\ref{l_pocet_orbit}, this action has at most seven orbits.
\end{proof}

\begin{corollary}\label{cor_trivial_symmetry_group}
If $S$ has two distinct indivisible edge-angles, then the symmetry group of $S$
(and of $c(S)$) is trivial.
\end{corollary}

\begin{proof}
By Lemma~\ref{lemma_at-least-4-edges}, $S$ has at least four edges of different
lengths for each of the two edge-angles. In particular, no symmetry can identify
any two of these eight edges and so the symmetry group of $S$ induces at least eight
orbits. It is therefore trivial by Lemma~\ref{lemma_7_orbits}. 
\end{proof}

Now assume for contradiction that $S$ has two indivisible edge-angles $\T_1$ and
$\T_2$. Let $T_1$ and $T_2$ be the corresponding triangles in $c(S)$. By
Lemma~\ref{lemma_at-least-4-edges}, each of $T_1, T_2$ occurs at least four times in
$c(S)$. 

We say that two edges of $c(S)$ are of the same {\em edge-type\/} if they have
equal labels; that is, they represent equal dihedral angles. An edge of type
$\alpha$ is also called an {\em $\alpha$-edge}. A triangle $T$ of $c(S)$ with
edges of types $\alpha, \beta, \gamma$ is called an {\em $(\alpha \beta
\gamma)$-triangle} and we write $T=(\alpha\beta\gamma)$.

\begin{observation}\label{obs_edge_type_twice}
Every edge of $c(S)$ belongs to a copy of the triangle $T_1$ or $T_2$. Moreover, every
edge-type of $T_1$ and $T_2$ 
occurs at least twice in $c(S)$. 
\end{observation}

\begin{proof}
The first part follows from the fact that every edge of $c(S)$ 
is contained in three triangles and at least eight of the ten
triangles of $c(S)$ are copies of $T_1$ or $T_2$.
The second claim follows again from
the fact that every edge of $c(S)$ is common to only three triangles.
\end{proof}

\begin{observation}\label{obs_multitype_4_times}
An edge-type common to both triangles $T_1, T_2$ occurs at least four times in
$c(S)$. Similarly, an edge-type occurring twice in $T_1$ (or $T_2$) occurs at
least four times in $c(S)$. 
\end{observation}

\begin{proof}
Let $\alpha$ be an edge-type common to both $T_1$ and $T_2$ and suppose that
each of $T_1$, $T_2$ has just one $\alpha$-edge. There are at least eight
triangles with an $\alpha$-edge in $c(S)$, therefore $c(S)$ has at least three
$\alpha$-edges. But if there are just three $\alpha$-edges, then some two of
them share a vertex (and hence a triangle). Therefore there are at most seven
triangles in $c(S)$ with exactly one $\alpha$-edge.

If $T_1$ has at least two $\alpha$-edges, then there are at least four pairs of
$\alpha$-edges in $c(S)$, hence at least four $\alpha$-edges.
\end{proof}

\begin{observation}\label{obs_4_times_6_times}
An edge-type $\alpha$ occurring four times together in $T_1$ and $T_2$ occurs at least six times in
$c(S)$.  
\end{observation}

\begin{proof}
Since each of the triangles $T_1$, $T_2$ has at least four copies in $c(S)$, the number of incidences of $\alpha$-edges with copies of triangles $T_1$ and $T_2$ in $c(S)$ is at least $16$. Since every edge forms at most three incidences, the observation follows.
\end{proof}

\begin{observation}\label{obs_common_edge_type}
The triangles $T_1$ and $T_2$ have at least one common edge-type.
\end{observation}

\begin{proof}
The observation follows from the fact that the the union of four different triangles in $c(S)$ has always at least six edges.
\end{proof}

By Observations~\ref{obs_edge_type_twice},~\ref{obs_multitype_4_times}~and~\ref{obs_common_edge_type}, the
edges of $c(S)$ have at most four types in total, since the common edge-type of
$T_1$ and $T_2$ occurs four times and every other edge-type occurs at least
twice. From these observations it also follows that if there are four different
edge-types, then three of them, $\beta,\gamma,\delta$, appear just once in $T_1$
or $T_2$ and the remaining one, $\alpha$, is common to $T_1$ and $T_2$ and appears twice in $T_1$ or twice in $T_2$. Similarly if
there are three different edge-types, then one of them appears at least three
times together in $T_1$ and $T_2$.

If there are just two different edge-types in $c(S)$, then $c(S)$ has a
non-trivial symmetry, which follows from the fact that every graph on five vertices has a nontrivial automorphism. But this contradicts Corollary~\ref{cor_trivial_symmetry_group}. 

Thus there are three or four different edge-types in $c(S)$ and we have 
five essentially different cases for the types of $T_1$ and $T_2$; see Table~\ref{table_two_indivisible}. Here by $\alpha,\beta,\gamma,\delta$ we denote pairwise different angles. By Observations~\ref{obs_edge_type_twice},~\ref {obs_multitype_4_times}~and~\ref{obs_4_times_6_times}, we can exactly determine the numbers of edges of each edge-type in $c(S)$. These are also shown in Table~\ref{table_two_indivisible}.

\begin{table}
\begin{center}
 \begin{tabular}{c|c|c|c|c|c|c}
 Case & type of $T_1$ & type of $T_2$ &  $\alpha$-edges & $\beta$-edges & $\gamma$-edges & $\delta$-edges \\
 \hline
 (1) & $(\alpha\alpha\beta)$  & $(\alpha\gamma\delta)$ & 4 & 2 & 2 & 2 \\
 (2) & $(\alpha\alpha\alpha)$ & $(\alpha\beta\gamma)$  & $6$ & $2$ & $2$ & 0 \\
 (3) & $(\alpha\alpha\beta)$  & $(\alpha\alpha\gamma)$ & $6$ & $2$ & $2$ & 0 \\
 (4) & $(\alpha\alpha\beta)$  & $(\alpha\gamma\gamma)$ & 4 & 2 & 4 & 0 \\
 (5) & $(\alpha\alpha\beta)$  & $(\alpha\beta\gamma)$  & 4 & 4 & 2 & 0 \\
\end{tabular}
\caption{Types of triangles $T_1$ and $T_2$ and the numbers of edges of each edge-type in $c(S)$.}
\label{table_two_indivisible}
\end{center}
\end{table}

In case $(1)$, since there are just two $\beta$-edges,
some two $(\alpha\alpha\beta)$-triangles in $c(S)$ share a $\beta$-edge. This
means that the $\alpha$-edges form a four-cycle. Further it follows that both
diagonals of the four-cycle are $\beta$-edges and that the fifth vertex of
$c(S)$ is joined by $\gamma$-edges to two opposite vertices of the four-cycle
and by $\delta$-edges to the other pair of opposite vertices; see
Figure~\ref{fig_no-2-indivisible}(a). 
This diagram has a $\mathbb{Z}_2\times\mathbb{Z}_2$ symmetry, which contradicts
Corollary~\ref{cor_trivial_symmetry_group}.

Now we consider case (2). Since
$K_4$ is the only graph with six edges and four triangles, the $\alpha$-edges
form a $K_4$ subgraph in $c(S)$, with two vertices joined by a $\beta$-edge and
two by a $\gamma$-edge to the remaining vertex of $c(S)$; see
Figure~\ref{fig_no-2-indivisible}(b).
Again, this diagram has a $\mathbb{Z}_2\times\mathbb{Z}_2$ symmetry, in contradiction
with Corollary~\ref{cor_trivial_symmetry_group}.

\begin{figure}
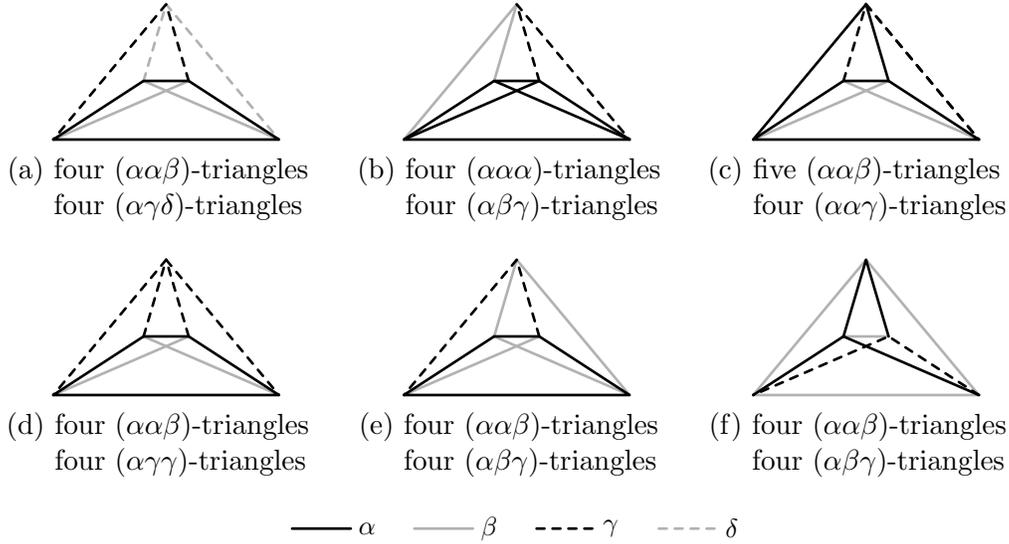

 \begin{center}
  \begin{tabular}{ccc}
\includegraphics{obrazky/coxeter.24} & \includegraphics{obrazky/coxeter.25} & \includegraphics{obrazky/coxeter.26}\\
(a) \parbox[t]{3.6cm}{four $(\alpha\alpha\beta)$-triangles\\ four $(\alpha\gamma\delta)$-triangles} & (b) \parbox[t]{3.6cm}{four
$(\alpha\alpha\alpha)$-triangles\\ four $(\alpha\beta\gamma)$-triangles} & (c)
\parbox[t]{3.6cm}{five $(\alpha\alpha\beta)$-triangles\\ four
$(\alpha\alpha\gamma)$-triangles}\\
& & \\
\includegraphics{obrazky/coxeter.27} & \includegraphics{obrazky/coxeter.28} & \includegraphics{obrazky/coxeter.281}\\
 (d) \parbox[t]{3.6cm}{four $(\alpha\alpha\beta)$-triangles\\ four
$(\alpha\gamma\gamma)$-triangles} & (e) \parbox[t]{3.6cm}{four
$(\alpha\alpha\beta)$-triangles\\ four $(\alpha\beta\gamma)$-triangles} & (f) \parbox[t]{3.6cm}{four
$(\alpha\alpha\beta)$-triangles\\ four $(\alpha\beta\gamma)$-triangles}\\
\end{tabular}
\begin{tabular}{c}
 \\
     \includegraphics{obrazky/popisky.201}  \\
\end{tabular}
  \caption{Coxeter diagrams for the case of two indivisible edge-angles.}
\label{fig_no-2-indivisible}
 \end{center}
\end{figure}

In case (3), let $H_{\alpha}$ be the subgraph of $c(S)$ formed by the $\alpha$-edges. Like in case (1), $H_{\alpha}$ contains a four-cycle. From the five possible extensions of the four-cycle by two edges, only $K_{2,3}$ has at least eight induced paths of length $2$. Thus $H_{\alpha}$ is isomorphic to $K_{2,3}$. The remaining edges form a disjoint
union of an edge and a triangle, so without loss of generality the two
$\gamma$-edges are contained in the triangle; see
Figure~\ref{fig_no-2-indivisible}(c).
This diagram has again a $\mathbb{Z}_2\times\mathbb{Z}_2$ symmetry, in contradiction
with Corollary~\ref{cor_trivial_symmetry_group}.

In case (4), just like in case (1), the $\alpha$-edges form a four-cycle whose diagonals are the two $\beta$-edges. The remaining edges are then $\gamma$-edges; see Figure~\ref{fig_no-2-indivisible}(d).
This diagram has a $D_4$ symmetry, in contradiction with
Corollary~\ref{cor_trivial_symmetry_group}.

Now we consider case (5). Out of the six subgraphs of $K_5$ with four edges, only the following three have four induced paths of length $2$: the star $K_{1,4}$, the four-cycle, and the \emph{fork}, which is the tree with the degree sequence $(3,2,1,1,1)$. If the $\alpha$-edges form $K_{1,4}$, there can be no $(\alpha\beta\gamma)$-triangles in $c(S)$. Thus, there are just two possibilities for the subgraph $H_{\alpha}$ of $c(S)$ formed by the $\alpha$-edges.

Suppose that $H_{\alpha}$ is a four-cycle. The diagonals of the four-cycle are then $\beta$-edges. In order to create four $(\alpha\beta\gamma)$-triangles, the vertices of the
four-cycle must be joined to the remaining vertex by two $\beta$-edges and two
$\gamma$-edges, in an alternating way; see Figure~\ref{fig_no-2-indivisible}(e). 
This diagram has a $\mathbb{Z}_2\times\mathbb{Z}_2$ symmetry, in contradiction
with Corollary~\ref{cor_trivial_symmetry_group}.

If $H_{\alpha}$ is a fork, the $\beta$-edges are uniquely determined, since the $\alpha$-edges form exactly four induced paths of length $2$.
The remaining two edges are $\gamma$-edges; see Figure~\ref{fig_no-2-indivisible}(f).
This diagram has a $\mathbb{Z}_2$-symmetry, in contradiction with
Corollary~\ref{cor_trivial_symmetry_group}.

We have finished the proof of the following statement.

\begin{proposition}\label{tvrzeni_no_2_edge_angles}
For $k \neq m^2$, every $k$-reptile four-dimensional simplex contains exactly one indivisible edge-angle.
\qed
\end{proposition}

\subsection{Basic facts and observations from spherical geometry} 

All spherical triangles are regarded as subsets of the $2$-dimensional unit sphere.
In this subsection we assume that 
$\T$ is a spherical triangle with angles
$\alpha\le\beta\le\gamma<\pi$ and corresponding opposite edges $a,b,c$.
The lengths of the edges are measured in radians and again denoted by $a,b,c$,
respectively. 

The following lemma lists a few standard facts about spherical triangles (see,
for example,~\cite{zwillinger}). The proof of part (a) can be found
in~\cite[Chapter 41]{pak}.
\bigskip

\begin{lemma}\label{lemma_fact} 
For a spherical triangle $\T$, we have
\begin{enumerate}
\item[{\rm(a)}] $\alpha+\beta+\gamma>\pi$, and $\alpha+\beta+\gamma - \pi$ is
equal to the spherical area $\Delta(\T)$ of $\T$. 

\item[{\rm(b)}] $\beta+\gamma<\pi + \alpha$; equivalently, $\Delta(\T)<2\alpha$ (spherical triangle inequality).

\item[{\rm(c)}] 
$\cos \gamma=-\cos \alpha \cos \beta + \sin \alpha \sin \beta \cos c$ (spherical law of cosines for angles).

\item[{\rm(d)}] If $\alpha < \beta < \gamma$, then $a < b < c$. If $\alpha = \beta < \gamma$, then $a = b < c$. If $\alpha < \beta = \gamma$, then $a < b = c$. 

\item[{\rm(e)}] $a < b + c$, $b < a + c$ and $c < a + b$ (triangle inequality
for the spherical distance).

\item[{\rm(f)}] $a,b,c < \pi$.

\end{enumerate}
\end{lemma}

The quantity $\alpha+\beta+\gamma - \pi$ is also called the {\em spherical
excess\/} of $\T$.

A {\em spherical lune\/} $\mathcal{L}$ with angle $\varphi < \pi$, which we shortly call the {\em $\varphi$-lune}, is a slice of
the sphere bounded by two great half-circles whose supporting planes have
dihedral angle $\varphi$. In other words, $\mathcal{L}$ is a spherical $2$-gon whose vertices are two antipodal points and both inner angles are equal to
$\varphi$. The spherical area of $\mathcal{L}$ is $2\varphi$. Note that $\mathcal{L}$ contains every spherical triangle with angle $\varphi$; this implies the spherical triangle inequality (Lemma~\ref{lemma_fact}(b)).

Consider a tiling of a lune $\mathcal{L}$ by spherical triangles. A tile $\T$ is called a {\em corner tile\/} if $\T$ shares a vertex $v$ with $\mathcal{L}$.
A tile $\T$ is called {\em corner-filling\/} if the complement of $\T$ in $\mathcal{L}$ is a spherical triangle. In particular, $\T$ shares a vertex $v$ with $\mathcal{L}$ and the two other vertices of $\T$ are internal points of the edges of $\mathcal{L}$. See Figure~\ref{obr_corner_filling}.

\begin{observation}\label{obs_corner}
Let $\varphi$ be the minimum angle of a spherical triangle $\T$ and let $f$ be the edge opposite to $\varphi$. Then in every tiling of the $\varphi$-lune by the copies of $\T$ there are two corner-filling tiles. Moreover, each of the corner-filling tiles neighbors with exactly one other tile, sharing the edge $f$.
\end{observation}

\begin{proof}
Since $\varphi$ is the minimum angle of $\T$, every corner tile must be corner-filling. By Lemma~\ref{lemma_fact}(f), a corner-filling tile contains only one vertex of the lune, hence there are at least two such tiles.
By Lemma~\ref{lemma_fact}(d), $f$ is the shortest edge of $\T$. The rest of the observation follows.
\end{proof}

\subsection{One indivisible edge-angle}

By Lemma~\ref{lemma_at-least-4-edges} and
Proposition~\ref{tvrzeni_no_2_edge_angles}, the simplex $S$ has only one
indivisible edge-angle $\T_0$. This means that all the remaining edge-angles of $S$ can by tiled with $\T_0$. In
particular, the spherical area of every
 spherical triangle representing an edge-angle of $S$ is an integer multiple of
the spherical area of $\T_0$. Let $T_0$ be the triangle in $c(S)$ corresponding to
the spherical triangle $\T_0$.

We say that a spherical triangle $\T$ has \emph{type $(\varphi\psi\chi)$} if its
internal angles are $\varphi,\psi$ and $\chi$; in this case we write $\T=(\varphi\psi\chi)$. 
We sometimes write the type of $\T$ as $(\varphi,\psi,\chi)$, to avoid confusion when substituting linear combinations of angles. We say that $\T$ has type
$(\varphi**)$ or $(\varphi\psi*)$ if it has type $(\varphi\psi\chi)$ for some angles $\psi,\chi$,
which may also be equal to $\varphi$ or to each other. Note that a spherical triangle
$\T=(\varphi\psi\chi)$ corresponds to a triangle $T=(\varphi\psi\chi)$ in
$c(S)$.

To simplify the notation, we will label the vertices of the Coxeter diagram $c(S)$ by $u,v,w,x,y$ instead of $F_1, \dots, F_5$. 

A Coxeter diagram $c(S)$ is {\em rich\/} if there is a triangle $T$ such that under the action of the symmetry group of $c(S)$ on the set of triangles in $c(S)$, the copies of $T$ form at least four different orbits. In this case we also say that the diagram $c(S)$ is $T$-{\em rich\/}. Lemma~\ref{lemma_at-least-4-edges} and Debrunner's
lemma imply the following important fact.

\begin{fact}
The Coxeter diagram of $S$ is $T_0$-rich.
\end{fact}

A spherical triangle $\T$ is {\em realizable} if $\T$ can be tiled with $\T_0$.
A triangle $T$ is {\em realizable} if its corresponding spherical triangle $\T$
is realizable.

The strategy of the proof is the following.

\begin{itemize}

\item Find all possible types of $\T_0$.

\item For every such triangle $\T_0$, let $\alpha$ be the minimal angle in
$\T_0$ (and $\beta$ the second minimal angle, if applicable). Investigate which
spherical triangles of type $(\alpha**)$ (or $(\beta**)$, if needed) are realizable.

\item Find all $T_0$-rich Coxeter diagrams whose all
$(\alpha**)$-triangles (and $(\beta**)$-triangles) are realizable.

\item Verify that such diagrams do not satisfy Fiedler's theorem.

\end{itemize}

We start with a simple observation about the Coxeter diagram of $S$.

\begin{observation}\label{obs_at_least_2_types}
The Coxeter diagram of $S$ has at least two different types of triangles.
\end{observation}

\begin{proof}
Suppose for contrary that all triangles in $c(S)$ are of the same type
$T=(\varphi_1\varphi_2\varphi_3)$. By double-counting, the numbers of
occurrences of the edge types in $c(S)$ are in the same ratio as in $T$. Since
the numbers $10$ and $3$ are relatively prime, it follows that
$\varphi_1=\varphi_2=\varphi_3$ and thus all dihedral angles in $S$ are equal.
But then $S$ is the regular simplex, which contradicts
Lemma~\ref{lemma_at-least-4-edges}.
\end{proof}


\subsubsection{Conditions on dihedral angles}

Here we prove several facts about the dihedral angles of $S$, which we use
further to restrict the set of possible types of the indivisible triangle
$\T_0$.

\begin{lemma}\label{lemma_o_vyskladani_cocky}
Let $\varphi_1, \varphi_2, \varphi_3$ be the angles of $\T_0$. Then for every $i \in \{1,2,3\}$, the spherical lune with angle $\varphi_i$ can be tiled with $\T_0$.
\end{lemma}

\begin{proof}
Assume that $\varphi_1 \le \varphi_2 \le \varphi_3$. Fix $i \in \{1,2,3\}$.
If $S$ is a $k$-reptile simplex for some $k>1$, then by induction, $S$ is
$k^n$-reptile for every $n\ge 1$. In particular, there is a tiling of $S$ with
simplices similar to $S$ where some of the tiles, $S'$, has an edge $e$ that is contained in the
interior of a $2$-face 
of $S$ with dihedral angle $\varphi_i$. 
Select an interior
point $x_i$ of $e$ that misses all vertices of all the tiles. Let $h_i$ be the hyperplane orthogonal to $e$ and containing $x_i$. In a small neighborhood of $x_i$ in $h_i$, the tiles with $x_i$ on their boundary induce a tiling of a wedge with angle $\varphi_i$ by triangular cones originating in $x_i$ and possibly by wedges with angles $\psi < \varphi_i$, where $\psi$ is an internal angle of some realizable triangle $\T$.

This is analogous to the situation in a tiling of a three-dimensional simplex in the neighborhood of an internal point $x$ of an edge such that $x$ is also a vertex of some tile. In the intersection of $h_i$ with a small sphere centered in $x_i$, we thus obtain a tiling of the spherical lune with angle $\varphi_i$ by spherical triangles corresponding to edge-angles of the tiles and possibly by spherical lunes with angles $\psi < \varphi_i$ corresponding to dihedral angles of the tiles. 
Since $\T$ can be tiled with $\T_0$, 
it follows that $\psi$ is a nonnegative integer combination of the angles $\varphi_j$ where $1\le j < i$. 
This means that the $\psi$-lune can be tiled by lunes with angles $\varphi_j$ where $1\le j < i$. 

Observe that since $x_i$ is an internal point of an edge of at least one tile, the tiling of the $\varphi_i$-lune contains at least one spherical triangle.
Thus for $i=1$, the tiling consists solely of realizable spherical triangles. For $i>1$, the $\varphi_i$-lune can be tiled with realizable spherical triangles and possibly $\varphi_j$-lunes with $j<i$. The lemma follows by induction on $i$.
\end{proof}

The following statement is a stronger variant of the Bricard's condition for
equidecomposable polyhedra~\cite{aigner_ziegler,benko}, \cite[Chapter
15]{pak}.

\begin{lemma}\label{lemma_bric}
Let $\varphi_1, \varphi_2, \varphi_3$ be the angles of $\T_0$. Then for every $i \in \{1,2,3\}$, there exist
nonnegative integers $m_{i,1}, m_{i,2}, m_{i,3}$ such that $m_{i,i} > 0$ and $m_{i,1}\varphi_1+m_{i,2}\varphi_2+m_{i,3}\varphi_3=\pi.$
\end{lemma}

\begin{figure}
 \begin{center}
   \includegraphics{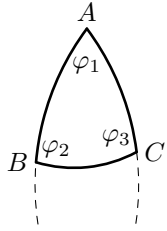}  
 \caption{A corner-filling tile in the $\varphi_1$-lune.}
 \label{obr_corner_filling}
 \end{center}
\end{figure}

\begin{proof}
Assume that $\varphi_1\le\varphi_2\le\varphi_3$. By Lemma~\ref{lemma_o_vyskladani_cocky}, the $\varphi_1$-lune $\mathcal{L}_1$ is tiled with $\T_0$. By Observation~\ref{obs_corner}, there is a corner-filling tile $\T_0^1$ whose vertices with inner angles $\varphi_2$ and $\varphi_3$ are internal points of the edges of $\mathcal{L}_1$; see Figure~\ref{obr_corner_filling}. In a small neighborhood of each of these two points we observe a tiling of the straight angle by the angles of $\T_0$, including $\varphi_2$ or $\varphi_3$, respectively. This shows the lemma for $i=2$ and $i=3$.

To show the lemma for $i=1$, we distinguish two cases. If $\varphi_1$ divides $\varphi_2$, the claim follows from the case $i=2$. Otherwise, we use Lemma~\ref{lemma_o_vyskladani_cocky} again, now for the $\varphi_2$-lune $\mathcal{L}_2$. The argument is analogous to the previous case since in the tiling of $\mathcal{L}_2$ with $\T_0$ each corner tile is corner-filling.
\end{proof}


For the rest of this section, let $\alpha$ be the minimum angle of $\T_0$. 

\begin{corollary}\label{cor_alfa_mensi_nez_pi_pul}
We have $\alpha < \pi/2$. 
\end{corollary}

\begin{proof}
Lemma~\ref{lemma_bric} implies that $\alpha \le \pi/2$. By Lemma~\ref{lemma_o_vyskladani_cocky}, the $\alpha$-lune is tiled with at least two copies of $\T_0$.
Now suppose that $\alpha = \pi/2$. 
Since $\Delta(\T_0)\ge \pi/2$ and the spherical area of the $\pi/2$-lune is $\pi$, the tiling of the $\pi/2$-lune consists of precisely two copies of $\T_0$ and $\Delta(\T_0) = \pi/2$. This means that $\T_0=(\pi/2,\pi/2,\pi/2)$. The angles of every spherical triangle tiled with $\T_0$ must be integer multiples of $\pi/2$, thus $\T_0$ is the only realizable triangle. This contradicts Observation~\ref{obs_at_least_2_types}.
\end{proof}

\begin{corollary}\label{cor_aaa}
We cannot have $\T_0=(\alpha\alpha\alpha)$.
\end{corollary}

\begin{proof}
If $\T_0=(\alpha\alpha\alpha)$, then by Lemma~\ref{lemma_bric}, $\alpha=\pi/n$ for some positive integer $n$. But we have $\alpha>\pi/3$ by Lemma~\ref{lemma_fact}(a) and $\alpha<\pi/2$ by Corollary~\ref{cor_alfa_mensi_nez_pi_pul}; a contradiction.
\end{proof}

We are left with three main cases for the type of the indivisible triangle $\T_0$, 
according to the symmetries and relative sizes of its angles. 

\begin{enumerate}
\item[(A)] Exactly two angles in $\T_0$ are equal, but not to $\alpha$. We write $\T_0=(\alpha\beta\beta)$. 
\item[(B)] Exactly two angles in $\T_0$ are equal to $\alpha$. We write $\T_0=(\alpha\alpha\beta)$. 
\item[(C)] All three angles in $\T_0$ are different. We write $\T_0=(\alpha\beta\gamma)$.
\end{enumerate}
 For the rest of this section we assume that $\alpha < \beta < \gamma$.


\subsubsection{\texorpdfstring
{Case (A): $\T_0=(\alpha\beta\beta)$.}{The case (alpha, beta, beta)}}

By Lemma~\ref{lemma_o_vyskladani_cocky}, the $\alpha$-lune $\mathcal{L}_{\alpha}$ can be tiled with $\T_0$. By Observation~\ref{obs_corner}, there is a corner-filling tile $\T_0^1$ sharing its shortest edge with another tile $\T_0^2$. See Figure~\ref{obr_corner_filling_abb}.
In the neighborhood of either common vertex of $\T_0^1$ and $\T_0^2$, we see the straight angle tiled with two angles $\beta$ and possibly other angles $\alpha$ or $\beta$.  Since $\alpha + 2\beta > \pi$ by Lemma~\ref{lemma_fact}(a), the two angles $\beta$ already tile the straight angle and hence $\beta=\pi/2$. It follows that two copies of $\T_0$ tile the whole $\alpha$-lune and thus $\T_0$ is the only realizable spherical triangle of type $(\alpha**)$.
 (We recall that $(\alpha * *)$ stands for $(\alpha \varphi \psi)$, where $\varphi, \psi$ may also be equal to $\alpha$ or to each other.)

This implies that the Coxeter diagram of $S$ has two
vertex-disjoint $\alpha$-edges, since it is $(\alpha\beta\beta)$-rich and has no $(\alpha\alpha*)$-triangle. Every remaining edge is adjacent to at least one $\alpha$-edge, hence it is a 
$\beta$-edge. The resulting diagram has only two orbits of $(\alpha\beta\beta)$-triangles and so it is not $(\alpha\beta\beta)$-rich, a contradiction.

\begin{figure}
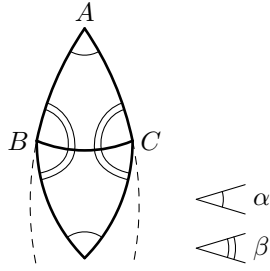

\begin{center}
  \begin{tabular}{ccc}
   \includegraphics{obrazky/sphere1.21}  & \includegraphics{obrazky/popisky.171}
  \end{tabular}
 \caption{A corner-filling $\alpha\beta\beta$-tile and its adjacent tile in the $\alpha$-lune.}
 \label{obr_corner_filling_abb}
 \end{center}
\end{figure}

\subsubsection{\texorpdfstring
{Case (B): $\T_0=(\alpha\alpha\beta)$.}{The case (alpha, alpha, beta)}} 

We start with an observation about $(\alpha\alpha\beta)$-rich Coxeter diagrams.

\begin{lemma}\label{lemma_dva_trojuhelniky_s_alphou}
There is at least one realizable spherical $(\alpha**)$-triangle different from
$\T_0$.
\end{lemma}

\begin{proof}
If $\T_0=(\alpha\alpha\beta)$ is the only realizable $(\alpha**)$ triangle, then
the $\alpha$-edges form a spanning complete bipartite subgraph in the Coxeter diagram of
$S$ and all the remaining edges are $\beta$-edges. The
$(\alpha\alpha\beta)$-triangles then form exactly two orbits, a contradiction.
\end{proof}

By Lemma~\ref{lemma_bric}, there exist integers $m_1 \ge 0$ and $m_2 \ge 1$ such that 
$m_1 \alpha + m_2\beta = \pi$. 
We distinguish two cases.

\begin{enumerate}
 \item[(1)] $m_1 \neq 0$. Since $2\alpha + \beta > \pi$ by Lemma~\ref{lemma_fact}(a), we have $m_1=m_2 = 1$ and thus $\alpha + \beta = \pi$.

\item[(2)] $m_1 = 0$. Then $\beta = \pi/m_2$. Since $3\beta > \pi$ and $\beta<\pi$, we have $m_2 = 2$ and so $\beta=\pi/2$. Now the inequality $2\alpha + \beta > \pi$ implies that $\alpha > \pi/4$. By Lemma~\ref{lemma_bric}, there exist integers $m'_1 \ge 1$ and $m'_2 \ge 0$ such that 
$m'_1 \alpha + m'_2\pi/2 = \pi$. Since $m'_2>0$ leads to contradiction, we have $\alpha = \pi/m'_1$. The only solution satisfying $\pi/4<\alpha<\pi/2$ is $\alpha = \pi/3$.
\end{enumerate}


\paragraph*{Case (1): $\alpha + \beta = \pi$.}
Let $\T=(\alpha\varphi\psi)$ be a realizable spherical $(\alpha**)$-triangle
different from $\T_0$ (with some of the angles possibly equal), whose existence is guaranteed by
Lemma~\ref{lemma_dva_trojuhelniky_s_alphou}. The spherical area of $\T$ satisfies
$\Delta(\T) \ge 2\Delta(\T_0)=2\cdot(2\alpha+\beta-\pi)=2\alpha$. But this contradicts the
spherical triangle inequality $\Delta(\T)<2\alpha$ (Lemma~\ref{lemma_fact}(b)).


\paragraph*{Case (2): $\alpha = \pi/3, \beta = \pi/2$.}

The only nonnegative integer combinations of $\alpha$ and $\beta$ that sum up to $\pi$ are $3\alpha$ and $2\beta$. This implies that in every tiling of a spherical polygon $\mathcal{P}$ by $\T_0$, for every internal point $x$ of an edge of $\mathcal{P}$, all incident tiles have the same angle at $x$. This somewhat restricts the set of possible tilings. Further restriction is obtained using the area argument.

The spherical area of $\T_0$ is $\beta+2\alpha-\pi=\pi/6$. The $\alpha$-lune, which contains every $(\alpha**)$-triangle, has spherical area $2\alpha=2\pi/3$. It follows that every
$(\alpha**)$-triangle is composed of at most three tiles. 

When constructing a tiling of an $(\alpha**)$-triangle, we always start with a corner-filling tile $\T_0^1$ of the $\alpha$-lune and then try to place additional tiles.
There is only one way of attaching a second tile to $\T_0^1$, yielding the triangle of type $(\alpha,\alpha,2\alpha)$; see Figure~\ref{obr_ACD}, left. Similarly, there is a unique way of attaching the third tile, which yields the triangle of type $(\alpha,2\alpha,\pi/2)$; see Figure~\ref{obr_ACD}, right. These are
the only realizable $(\alpha**)$-triangles other than $\T_0$.

\begin{figure}
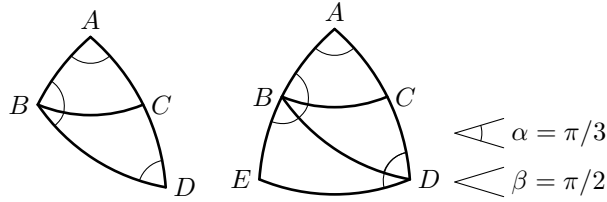

 \begin{center}
  \begin{tabular}{ccc}
     \includegraphics{obrazky/sphere1.18}  & \includegraphics{obrazky/sphere1.22}
     \includegraphics{obrazky/popisky.30}  \\
  \end{tabular}
   \caption[The $(\alpha\alpha2\alpha)$-triangle and the $(\alpha,2\alpha,\pi/2)$-triangle.]{The $(\alpha\alpha2\alpha)$-triangle and the $(\alpha,2\alpha,\pi/2)$-triangle composed of two and three $(\alpha\alpha\beta)$-tiles, respectively.}
 \label{obr_ACD}
 \end{center}
\end{figure}

Realizable $(\beta**)$-triangles can be composed from at most five tiles, as
their spherical area is smaller than $2\beta=\pi$. 
To construct a $(\beta**)$-triangle, we start with a corner tile $\T_0^1$ in the $\beta$-lune.
Since $\alpha$ does not divide $\beta$, the corner tile is corner-filling.
By Lemma~\ref{lemma_fact}(e) and by the symmetry of $\T_0$, the longest edge of $\T_0$ cannot be tiled with the shorter edges. This means that there is just one possible way of attaching another tile, $\T_0^2$, to $\T_0^1$; see Figure~\ref{fig_aab}, where $\T_0^1$ has vertices $ABC$ and $\T_0^2$ has vertices $BCF$.
The two tiles do not form a triangle, yet. Hence, there is at least one more tile  $\T_0^3$ adjacent to, say, $B$. The orientation of $\T_0^3$ where $\T_0^3$ shares the edge $BF$ with $\T_0^2$ gives an $(\alpha,2\alpha,\pi/2)$-triangle obtained earlier. The other orientation of $\T_0^3$, where the longest edge of $\T_0^3$ partially coincides with the edge $BF$, forces a fourth tile sharing the edge $CF$ with $\T_0^2$, forming an $(\alpha,2\alpha,\pi/2)$-triangle $ABE$ with $\T_0^1$ and $\T_0^2$. The remaining uncovered part of the edge $BE$ is shorter than all edges of $\T_0$, thus such a tiling cannot be completed to a $(\beta**)$-triangle. 

To extend the $(\alpha,2\alpha,\pi/2)$-triangle $ACD$, at least two more tiles are needed. There is precisely one way of attaching two more tiles, giving a $(\beta,2\alpha,2\alpha)$-triangle composed of five pieces; see Figure~\ref{fig_aab}. Therefore, the only $(\beta**)$-triangles are $(\alpha\alpha\beta)$, $(\alpha,2\alpha,\beta)$ and  $(\beta,2\alpha,2\alpha)$.
In particular, there is no $(\beta\beta*)$-triangle.

This implies that the Coxeter diagram of $S$ has exactly two vertex-disjoint 
$\beta$-edges. Let $uv$ and $xy$ be the two $\beta$-edges and let $w$ be the fifth vertex of $c(S)$. If both triangles $uvw$ and $xyw$ are $(\alpha\alpha\beta)$-triangles, then all edges incident with $w$ are $\alpha$-edges. No other edge can be an $\alpha$-edge, since the $(\alpha\alpha\alpha)$-triangle is not realizable. In other words, the triangles $uvw$ and $xyw$ are the only $(\alpha\alpha\beta)$-triangles; a contradiction. Hence at least three of the triangles induced by the vertices $u,v,x,y$ are $(\alpha\alpha\beta)$-triangles. But then all four of the triangles are $(\alpha\alpha\beta)$-triangles since the edges $ux,uy,vx,vy$ must be $\alpha$-edges. Every triangle containing one of these four $\alpha$-edges and the vertex $w$ must be of type $(\alpha,\alpha,2\alpha)$. Therefore, without loss of generality, the edges $uw$ and $vw$ are $\alpha$-edges, and the edges $xw$ and $yw$ are $2\alpha$-edges. A similar diagram is displayed in Figure~\ref{fig_no-2-indivisible}c), where instead of $2\alpha$-edges we have $\gamma$-edges. But in such 
a diagram, there are only three orbits of $(\alpha\alpha\beta)$-triangles; a contradiction.

\begin{figure}
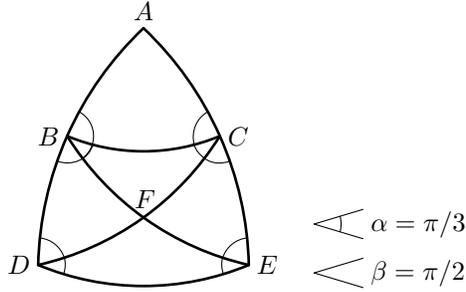

 \begin{center}
   \begin{tabular}{ccc}
  \includegraphics{obrazky/sphere1.11} & \includegraphics{obrazky/popisky.30} 
  \end{tabular}
  \caption[The $(\beta2\alpha2\alpha)$-triangle.]{The $(\beta2\alpha2\alpha)$-triangle composed of five
$(\alpha\alpha\beta)$-tiles with $\alpha=\pi/3, \beta=\pi/2$.}
 \label{fig_aab}
 \end{center}
\end{figure}

\subsubsection{\texorpdfstring
{Case (C): $\T_0=(\alpha\beta\gamma)$.}{The case (alpha, beta, gamma)}} 

First we obtain some more information about the Coxeter diagram of $S$.

\begin{lemma}\label{at-least-2} 
There is at least one realizable spherical $(\alpha**)$-triangle different from
$\T_0$.
The same is true for triangles of type $(\beta**)$ and $(\gamma**)$.
\end{lemma}

\begin{proof}
Assume for contradiction that $\T_0$ is the only realizable spherical triangle of
type $(\alpha**)$. It follows
that there are exactly two vertex-disjoint $\alpha$-edges in $c(S)$. Since every
other edge in $c(S)$ is adjacent to an $\alpha$-edge, all edges in $c(S)$ are $\beta$-edges or
$\gamma$-edges, and $c(S)$ is isomorphic to one of the two diagrams in Figure~\ref{fig_a-b-c}.
In both diagrams, the six $(\alpha\beta\gamma)$-triangles form only three orbits; a contradiction.
\end{proof}

\begin{figure}
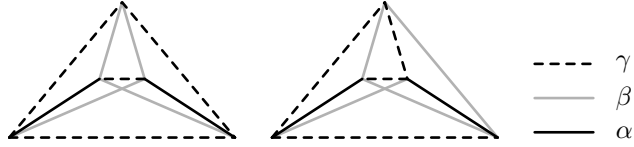

 \begin{center}
  \begin{tabular}{ccc}
   \includegraphics[scale=1]{obrazky/coxeter.29} & \includegraphics[scale=1]{obrazky/coxeter.30} &
\includegraphics{obrazky/popisky.9}\\
  \end{tabular}
 \end{center}
\caption{Two possible Coxeter diagrams if $\T_0$ is the only triangle of type
$(\alpha**)$.}
\label{fig_a-b-c}
\end{figure}

By $P_k$ we denote a path with $k$ vertices and by $P_k+P_l$ a disjoint union of paths.

\begin{lemma}\label{lem_partial_classification_cox_2_3_alpha_edges}
The Coxeter diagram of $S$ has two or three $\alpha$-edges and they form a subgraph isomorphic to $P_2+P_2$ or $P_2+P_3$.
\end{lemma}

\begin{proof} 
Let $H=c(S)$ be the Coxeter diagram of $S$. Let $V(H)=\{u,v,w,x,y\}$ and let $H_{\alpha}=(V(G),E_{\alpha})$ be the subgraph of $H$ formed by the $\alpha$-edges. Since $H$ is $(\alpha\beta\gamma)$-rich, it has at least four $(\alpha\beta\gamma)$-triangles, and hence it has at least two $\alpha$-edges, at least two $\beta$-edges and at least two $\gamma$-edges.

Suppose that $|E_{\alpha}|= 2$. If the two $\alpha$-edges are adjacent, say, $E_{\alpha}=\{uv,uw\}$, then all the triangles $uvx,uvy,uwx,uwy$ are of type $(\alpha\beta\gamma)$. In particular, the edges $vx$ and $wx$ are of the same type, either $\beta$ or $\gamma$, and the edges $vy$ and $wy$ are of the same type as well. Therefore, $H$ has a symmetry switching $v$ with $w$, and so there are at most two orbits of $(\alpha\beta\gamma)$-triangles; a contradiction. It follows that $H_{\alpha}$ is a matching.

Suppose that $|E_{\alpha}|= 3$. If $H_{\alpha}$ is isomorphic to the star $K_{1,3}$, say, $E_{\alpha}=\{xu,xv,xw\}$, then every $(\alpha\beta\gamma)$-triangle must contain the vertex $y$, so there can be at most three such triangles. If $H_{\alpha}$ is isomorphic to the path $P_4$, say, $E_{\alpha}=\{xu,uv,vw\}$, then the edges $xv$ and $uw$ cannot have type $\beta$ or $\gamma$, since the spherical triangles of type $(\alpha\alpha\beta)$ and $(\alpha\alpha\gamma)$ have smaller area than $\T_0$ and so they are not realizable. This again implies that every $(\alpha\beta\gamma)$-triangle must contain the vertex $y$ and so there are at most three of them. Also, $H_{\alpha}$ cannot form a triangle, since the spherical triangle of type $(\alpha\alpha\alpha)$ is not realizable. This leaves only one option: $H_{\alpha}$ forms a subgraph isomorphic to $P_2+P_3$.

Suppose that $|E_{\alpha}|\ge 4$. If $H_{\alpha}$ contains a star $K_{1,4}$, then no other edge can be of type $\beta$ or $\gamma$. If $H_{\alpha}$ contains a fork, say, $E_{\alpha}\supseteq\{uv,vw,wx,wy\}$, then only two edges, $ux$ and $uy$, can be of type $\beta$ or $\gamma$. If $H_{\alpha}$ contains a path $P_5$, say, $E_{\alpha}\supseteq\{uv,vw,wx,xy\}$, then only three edges, $ux,uy$ and $vy$, can be of type $\beta$ or $\gamma$. Since $H_{\alpha}$ cannot contain triangles, the only remaining possibility is that $H_{\alpha}$ is isomorphic to the $4$-cycle, say, $E_{\alpha}=\{uv,vx,xy,yu\}$. All edges of type $\beta$ or $\gamma$ must be incident with $w$, hence $uw$ and $xw$ are of the same type, and also $vw$ and $yw$ are of the same type. Regardless of the type of the diagonals $ux$ and $vy$, this diagram has a symmetry group generated by the transpositions $(u,x)$ and $(v,y)$, and so the $(\alpha\beta\gamma)$-triangles form just one orbit; a contradiction.
\end{proof}

In the following lemma we obtain some partial information about the angles of $\mathcal{T}_0$ and identify two basic cases.

\begin{lemma}\label{gamma=pi/2}
 If $\T_0 = (\alpha\beta\gamma)$ then $\gamma = \pi/2$. Furthermore,
\begin{enumerate} 
\item[a)]$\alpha + 2\beta = \pi$, or 
\item[b)]$\beta = \pi/3$ and $\alpha > \pi/6$.
\end{enumerate}
\end{lemma}

\begin{proof}
By Lemma~\ref{at-least-2}, the spherical area of the $\alpha$-lune $\mathcal{L}_{\alpha}$ is greater than $2\Delta(\T_0)$. By Lemma~\ref{lemma_o_vyskladani_cocky}, there is a tiling of the lune $\mathcal{L}_{\alpha}$ by at least three copies of $\T_0$.    

Let $\T_0^1$ be a corner-filling tile with vertices $A,B,C$ incident with angles $\alpha, \beta, \gamma$, respectively. In particular, $A$ is a vertex of $\mathcal{L}_{\alpha}$.
By Observation~\ref{obs_corner}, $\T_0^1$ is adjacent to a tile $\T_0^2$ with vertices $B,C,D$, which can be placed in two possible orientations; see Figure~\ref{obr_abc_2dlazdice}.
If $\T_0^2$ has the same orientation as $\T_0^1$, the quadrilateral $ABDC$ is a spherical parallelogram with angles $\beta+\gamma$ at vertices $B$ and $C$. By Lemma~\ref{at-least-2}, two copies of $\T_0$ cannot tile $\mathcal{L}_{\alpha}$ and so $\beta+\gamma<\pi$. Since $\alpha + \beta + \gamma > \pi$ by Lemma~\ref{lemma_fact}(a), the parallelogram $ABDC$ cannot be completed to a tiling of $\mathcal{L}_{\alpha}$. Therefore $\T_0^2$ and $\T_0^1$ have opposite orientations. 

\begin{figure}
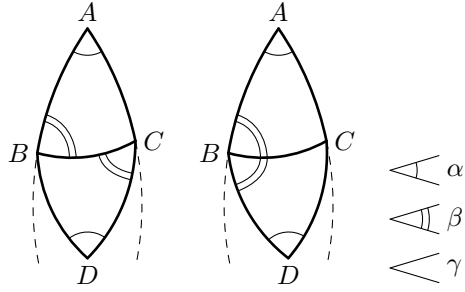

 \begin{center}
  \begin{tabular}{ccc}
   \includegraphics{obrazky/sphere1.3} & \includegraphics{obrazky/sphere1.19} & \includegraphics{obrazky/popisky.20}
  \end{tabular}
 \end{center}
\caption{Two possibilities for the first two tiles in a tiling of $\mathcal{L}_{\alpha}$ by $\T_0$.}
\label{obr_abc_2dlazdice}
\end{figure}

Since $\alpha+2\gamma>\alpha+\beta+\gamma>\pi$, no other tile can be incident to $C$ and so $\gamma = \pi/2$.
That is, the two tiles $\T_0^1$ and $\T_0^2$ form a triangle $ABD$. The angles of the tiles incident to $B$ include two angles $\beta$, and together they sum up to $\pi$. No tile can have angle $\gamma$ at $B$ since $2\beta+\gamma>\pi$.
Therefore, there exist non-negative integers $n_1$ and $n_2\ge 2$ such that 
$n_1\alpha + n_2\beta = \pi$. 
Since $\gamma = \pi/2$, we have $\alpha + \beta > \pi/2$, implying $n_1\le 1$ and $\beta > \pi/4$. There are 
only two cases: 
\begin{enumerate}
 \item[a)] $n_1 = 1$: then $n_2=2$ and thus $\alpha + 2\beta = \pi$.
 \item[b)] $n_1 = 0$: then $\beta = \pi/3$ and consequently $\alpha > \pi/6$.
\end{enumerate}
This concludes the proof.
\end{proof}

Now we deal separately with the two cases from Lemma~\ref{gamma=pi/2}.

\paragraph{Case a) $\alpha + 2\beta = \pi$.}

\begin{figure}
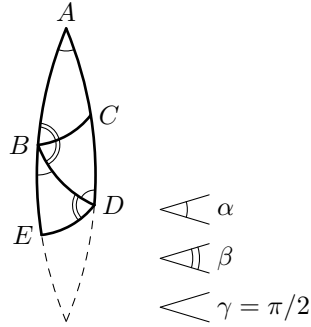

 \begin{center}
  \begin{tabular}{cc} 
   \includegraphics{obrazky/sphere-troj.0} & \includegraphics{obrazky/popisky.170}\\
  \end{tabular}
 \end{center}
\caption{Tiling of spherical triangles for $\gamma = \pi/2$ and $\alpha + 2\beta = \pi$.}
\label{fig_2beta}
\end{figure}

By Observation~\ref{obs_corner} and by the fact that $\beta+\gamma<\pi<\alpha+\beta+\gamma$, the tiling of every realizable $(\alpha**)$-triangle other than $\T_0$ contains two tiles $ABC$ and $BCD$ with opposite orientations as in Figure~\ref{fig_2beta},  forming a triangle of type $(\alpha,\alpha,2\beta)$. By the triangle inequality (Lemma~\ref{lemma_fact}(e)), the edge $BD$ cannot be subdivided by the edge of $\T_0$ opposite to $\beta$, thus there is just one possible way of placing a third tile: the triangle $BDE$ in Figure~\ref{fig_2beta}. These three tiles form a triangle of type $(\alpha,\alpha+\beta,\gamma)$. The fourth tile would fill the whole $\alpha$-lune, therefore the only realizable  $(\alpha**)$-triangles are  $(\alpha\beta\gamma)$, $(\alpha,\alpha,2\beta)$ and $(\alpha,\alpha+\beta,\gamma)$.

\begin{lemma}\label{lem_abc_aa2b_aa+bc}
Every $(\alpha\beta\gamma)$-rich Coxeter diagram with five vertices where all $(\alpha**)$-triangles are of type $(\alpha\beta\gamma)$, $(\alpha,\alpha,2\beta)$ or $(\alpha,\alpha+\beta,\gamma)$, is isomorphic to one of the five diagrams in Figure~\ref{fig_alfa+2beta=pi}.
\end{lemma}

\begin{proof}
Let $H$ be a Coxeter diagram satisfying the assumptions of the lemma. Let $V(H)=\{u,\allowbreak v,\allowbreak w,\allowbreak x,\allowbreak y\}$ and let $E_{\alpha}$ be the set of $\alpha$-edges. By Lemma~\ref{lem_partial_classification_cox_2_3_alpha_edges}, we distinguish two cases, up to isomorphism.

1) $E_{\alpha}=\{ux,vy\}$. In this case, all triangles containing an $\alpha$-edge are of type $(\alpha\beta\gamma)$ or $(\alpha,\alpha+\beta,\gamma)$. In particular, every such triangle contains exactly one $\gamma$-edge. By symmetry, we may assume that $xw$ and $yw$ are $\gamma$-edges. The other two $\gamma$-edges form a matching on vertices $u,v,x,y$; there are two possibilities, $\{uv,xy\}$ and  $\{uy,vx\}$. The remaining four edges are of type $\beta$ or $\alpha+\beta$.

If all the remaining edges are $\beta$-edges, or if both $uw$ and $vw$ are of type $\alpha+\beta$, the diagram has a symmetry $\Phi$ exchanging simultaneously $u$ with $v$ and $x$ with $y$, so the $(\alpha\beta\gamma)$-triangles form at most three orbits. Since there are at least four $(\alpha\beta\gamma)$-triangles in $H$, it follows that exactly one of the four remaining edges is of type $\alpha+\beta$ and the remaining three edges are $\beta$-edges. If $uv$ or $xy$ is the edge of type $\alpha+\beta$ (so $uy$ and $vx$ were chosen as $\gamma$-edges), then $\Phi$ is again a symmetry of the diagram and the $(\alpha\beta\gamma)$-triangles form only two orbits. 
Thus, up to isomorphism, we have three possibilities for $H$; see Figure~\ref{fig_alfa+2beta=pi}(a)--(c). 

2) $E_{\alpha}=\{uw,vw,xy\}$. In this case the edge $uv$ must be of type $2\beta$ and it is the only edge of this type. Each of the seven $(\alpha**)$-triangles other than $uvw$ has exactly one $\gamma$-edge. Due to symmetry, we may assume that $xw$ is a $\gamma$-edge. Then $yu$ and $yv$ must be $\gamma$-edges as well. The remaining three edges, $yw,xu$ and $xv$, are of type $\beta$ or $\alpha+\beta$. If all these three edges are of type $\beta$, the diagram has a symmetry $\Psi$ exchanging $u$ and $v$, but the $(\alpha\beta\gamma)$-triangles still form four orbits; see Figure~\ref{fig_alfa+2beta=pi}(d). Otherwise, exactly one of the three edges $yw,xu,xv$ is of type $\alpha+\beta$. If $yw$ is of type $\alpha+\beta$, then due to the symmetry $\Psi$, the $(\alpha\beta\gamma)$-triangles form only two orbits. The other two cases give isomorphic diagrams; see Figure~\ref{fig_alfa+2beta=pi}(e).
\end{proof}

\begin{figure}
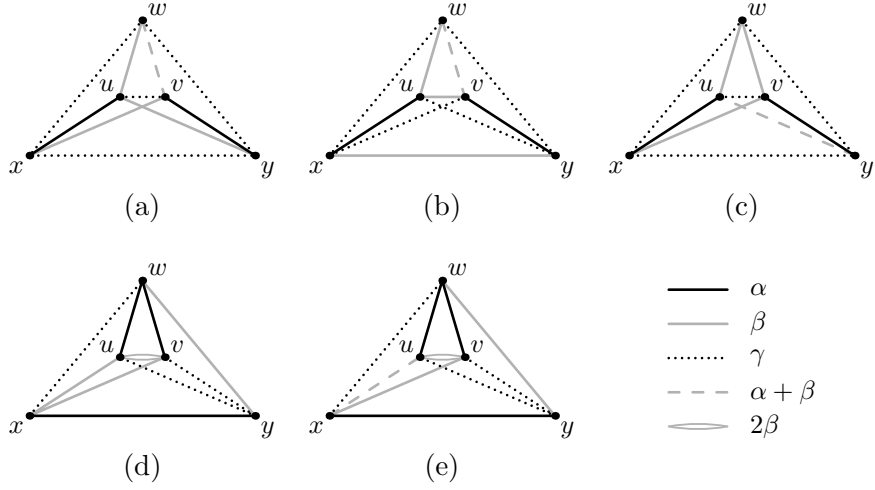

 \begin{center}
  \begin{tabular}{ccc}
\includegraphics{obrazky/coxeter.600} & \includegraphics{obrazky/coxeter.601} & \includegraphics[scale=1]{obrazky/coxeter.602}\\
(a) & (b) & (c) \\
 & & \\
\includegraphics[scale=1]{obrazky/coxeter.603} & \includegraphics[scale=1]{obrazky/coxeter.604} & \includegraphics{obrazky/popisky.10} \\
 (d) & (e) & 
\end{tabular}
\caption{Coxeter diagrams for the case $\alpha + 2\beta = \pi$.}
\label{fig_alfa+2beta=pi}
 \end{center}
\end{figure}

We immediately notice that the diagram in Figure~\ref{fig_alfa+2beta=pi}(e) cannot be a Coxeter diagram of $S$, since $xuv$ is a $(\beta,\alpha+\beta,2\beta)$-triangle but there is no spherical triangle of type $(\beta,\alpha+\beta,2\beta)$ by Lemma~\ref{lemma_fact}(b).

We are left with the diagrams in Figure~\ref{fig_alfa+2beta=pi}(a)--(d). Since investigating realizable $(\beta**)$-triangles does not seem to help much, we proceed to the next step and use Fiedler's theorem (Theorem~\ref{theorem_fiedler}).

Recall that the matrix $A$ associated to a simplex $S$ satisfies $a_{i,i}=-1$ and $a_{ij} = \cos\beta_{ij}$ for $i\neq j$, where $\beta_{ij}$ is the dihedral angle between facets $F_i$ and $F_j$. In particular, the matrix $A$ is completely determined by the Coxeter diagram $c(S)$.

Let $t:=\cos\beta$ and $s:=\cos\alpha$. Since $\alpha + 2\beta = \pi$ and
$\alpha < \beta$, it follows that $s= 1-2t^2$ and $t \in (0,1/2)$. Moreover, $\cos 2\beta = 2t^2-1$ and $\cos(\alpha + \beta) = \cos(\pi-\beta) = -t.$

The following matrices $A_1, \dots, A_4$ are associated to the simplices represented by the diagrams in Figure~\ref{fig_alfa+2beta=pi}(a)--(d), respectively.
The rows and columns  of $A_1,\dots,A_4$ are indexed by $u,v,w,x,y$ (in this order), where the vertices $u,v,w,x,y$ of $c(S)$ represent facets $F_1,\dots,F_5$ in $S$.

$$A_1=
\left(\begin{array}{rrrrr}
-1 & 0 & t & 1-2t^{2} & t\\
0 & -1 & -t & t & 1-2t^{2}\\
t & -t & -1 & 0 & 0\\
1-2t^{2} & t & 0 & -1 & 0\\
t & 1-2t^{2} & 0 & 0 & -1
\end{array}\right)
$$

$$A_2=
\left(\begin{array}{rrrrr}
-1 & t & t & 1-2t^{2} & 0 \\
t & -1 & -t & 0 & 1-2t^{2} \\
t & -t & -1 & 0 & 0 \\
1-2t^{2} & 0 & 0 & -1 & t \\
0 & 1-2t^{2} & 0 & t & -1
\end{array}\right)
$$

$$ 
A_3 = 
\left(\begin{array}{rrrrr}
-1 & 0 & t & 1-2t^{2} & -t \\
0 & -1 & t & t & 1-2t^{2} \\
t & t & -1 & 0 & 0 \\
1-2t^{2} & t & 0 & -1 & 0 \\
-t & 1-2t^{2} & 0 & 0 & -1
\end{array}\right)$$

$$
A_4 = 
\left(\begin{array}{rrrrr}
-1 & 2t^{2}-1 & 1-2t^{2} & t & 0 \\
2t^{2} - 1 & -1 & 1-2t^{2} & t & 0 \\
1-2t^{2} & 1-2t^{2} & -1 & 0 & t \\
t & t & 0 & -1 & 1-2t^{2} \\
0 & 0 & t & 1-2t^{2} & -1
\end{array}\right)
$$

%
%
%

Considering $t$ as a variable, the determinants of the matrices $A_1, \dots, A_4$ are polynomials in $t$. Let $\Gamma(A_i)$ be the set of real roots of the determinant of $A_i$. Rounding the roots to two decimal places, 
we have:

\bigskip
\noindent
{\small
\begin{center}
\begin{tabular}{c|c|c}
 $A$ & $\det(A)$ & $\Gamma(A)$\\
 \hline
 
 $A_1$ & $-t^2(2t-1)(2t^2 - t - 2)(4t^3 + 4t^2 - t - 2)$ & $\{0,0.5, -0.78,1.28, 0.63\}$ \\
  $A_2$ & $-t^2(2t - 1)(2t^2 + t - 2)(4t^3 + 2t^2 - 3t - 2)$ & $\{0,0.5, -1.28, 0.78, 0.92\} $\\
   $A_3$ & $-t^4(2t - 1)(2t + 1)(4t^2- 3)$ & $\{0,\pm 0.5, \pm 0.87\}$ \\
    $A_4$ & $-8t^4(2t^2 - 1)(4t^4 - 7t^2 + 2)$ & $\{0,\pm 0.71, \pm1.18, \pm 0.60\}$
\end{tabular}
\end{center}
}
\bigskip
\medskip

By Fiedler's theorem, the matrix associated to a simplex is singular. Therefore, the determinant of $A_i$ must have a root in the interval $(0,1/2)$. Since no $A_i$ satisfies this condition, we have a contradiction and the lemma follows.

\paragraph{Case b) $\beta = \pi/3$ and $\alpha > \pi/6$.}

\begin{lemma}\label{lemma_4-poss-for-alfa}
If $\T_0=(\alpha,\pi/3,\pi/2)$ with $\alpha>\pi/6$, then $\alpha \in\{\pi/4, 2\pi/9, \pi/5\}$.
\end{lemma}

\begin{proof}
By Lemma~\ref{lemma_bric}, there are integers 
$m \ge 1$ and $n,p\ge 0$  such that $m\alpha + n\pi/3 +p\pi/2 = \pi$. If $p\ge 1$, then $p=1$ and $n=0$. Since $\pi/6 <\alpha < \pi/3$, we have $m=2$ and thus $\alpha=\pi/4$. If $p=0$ then $n \le 1$. 
For $n=0$ we have $\alpha = \pi/m$, thus $m = 4$ or $5$. For $n=1$ we have 
$\alpha=2\pi/(3m)$, which is in the interval $(\pi/6,\pi/3)$ only for $m = 3$. The lemma follows. 
\end{proof}

The following simple observation will be useful for determining all realizable $(\alpha**)$- and $(\pi/3**)$-triangles.

\begin{observation}\label{obs_dlazdeni_cocky}
 Let $\mathcal{L}$ be a $\varphi$-lune that can be tiled with $\T_0$. Let $\mathcal{H}$ be a realizable spherical triangle whose two copies tile the lune $\mathcal{L}$ and let $\mathcal{K} \subseteq \mathcal{H}$ be a realizable corner-filling spherical
 triangle in $\mathcal{L}$ whose tiling by copies of $\T_0$ can be extended to a tiling of $\mathcal{H}$. Then the complement of $\mathcal{K}$ in $\mathcal{L}$ is also realizable; see Figure~\ref{fig_dlazdeni_cocky}. 
\qed
\end{observation}

\begin{figure}
 \begin{center}
    \includegraphics[scale=1]{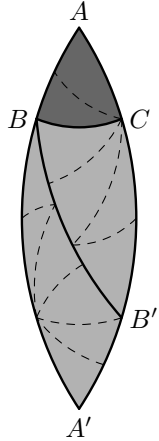}\\
   \caption{A tiling of a lune $\mathcal{L}$ where $\mathcal{H} = ABB'$ and $\mathcal{K} = ABC$. }
\label{fig_dlazdeni_cocky}
 \end{center}
\end{figure}

In the following lemma we investigate all
realizable $(\alpha**)$- and $(\beta**)$-triangles
for $\T_0=(\alpha,\pi/3,\pi/2).$ For better clarity we write $\alpha$ and
$\beta$ rather than their numerical values.

\begin{lemma}\label{l_pi-4-5-9}
 Let $\T_0 = (\alpha,\beta,\pi/2)$, where $\beta=\pi/3$. Depending on the value of $\alpha$, the realizable $(\alpha**)$-triangles 
and $(\beta**)$-triangles other than $\T_0$ are the following.
\begin{enumerate}

\item[{\rm 1)}]
  $\alpha = \pi/4$
 \begin{itemize}
  \item $(\alpha,\alpha,2\beta)$, $(\alpha,\pi/2,\pi/2)$, $(\alpha,\beta,3\alpha)$,
$(\alpha,\pi/2,2\beta)$
  \item $(\beta,\beta,\pi/2)$, $(\beta,\beta,2\beta)$, $(\beta,\pi/2,2\beta)$, $(\beta,\pi/2,3\alpha)$
 \end{itemize}

\item[{\rm 2)}]
 $\alpha = \pi/5$
 \begin{itemize}
  \item $(\alpha,\alpha,2\beta)$, $(\alpha,2\alpha,\pi/2)$, $(\alpha,\beta,3\alpha)$, $(\alpha, \beta, 2\beta)$, $(\alpha,\alpha,4\alpha)$,  
  $(\alpha,2\alpha,2\beta)$,\\ $(\alpha,\pi/2,3\alpha)$, $(\alpha,\beta,4\alpha)$, $(\alpha,\pi/2,2\beta)$
  \item $(\beta,\beta,2\alpha)$, $(\beta,2\alpha,\pi/2)$, $(\beta,2\alpha,3\alpha)$,
$(\beta,\pi/2,3\alpha)$, $(\beta,\beta,4\alpha)$, $(\beta,2\alpha,4\alpha)$,\\ $(\beta,3\alpha,2\beta)$, $(\beta,\pi/2,4\alpha)$
 \end{itemize}

\item[{\rm 3)}]
 $\alpha = 2\pi/9$
 \begin{itemize}
  \item $(\alpha,\alpha,2\beta)$, $(\alpha,2\alpha,\pi/2)$, $(\alpha,\beta,2\beta)$, 
$(\alpha,\pi/2,\alpha+\beta)$, $(\alpha,\beta,2\alpha+\beta)$, $(\alpha, \alpha, 4\alpha)$, $(\alpha,\pi/2,2\beta)$
  \item $(\beta,\beta,2\alpha)$, $(\beta,\alpha+\beta,2\beta)$, $(\beta,\pi/2,2\alpha+\beta).$
 \end{itemize}
\end{enumerate}
\end{lemma}

\begin{proof}
Let $a,b,c$ denote the edges (and also their lengths) of $\T_0$ opposite to angles $\alpha,\beta,\gamma$, respectively.

Now we describe a general procedure of finding all realizable $(\tau**)$-triangles, which we then apply to specific values of $\tau$.
First we generate all spherical triangles whose angles and edges can be obtained as nonnegative integer combinations of angles and edges, respectively, of the basic tile $\T_0$, and whose spherical area is a multiple of $\Delta(\T_0)$, the spherical area of $\T_0$. After that we check whether they are realizable. Clearly, this involves going through only a finite number of triangles and a finite number of tilings. In fact, it will turn out that all the generated spherical triangles but one are realizable.

Here we provide more details about the procedure.
Let $\T=(\tau\varphi\psi)$, $\T \neq \T_0$, be a realizable triangle.
We may assume that $\T$ is a corner-filling triangle in the $\tau$-lune and $\varphi \le \psi$.

Since $\T$ is realizable, it follows that $\Delta(\T) = n \Delta(\T_0)$, where 
 $n$ is the number of copies of $\T_0$ needed to tile $\T$. Moreover, the angles $\tau, \varphi, \psi$ can be expressed as nonnegative integer combinations of $\alpha,\beta,\gamma$.
 By Lemma~\ref{lemma_fact}(a), we have 
\begin{align}
 n\Delta(\T_0) &=  \Delta(\T) = \tau + \varphi + \psi  - \pi      \nonumber \\
 n(\alpha + \beta + \gamma - \pi) + \pi -\tau &=  \varphi + \psi  \nonumber \\
 n(\alpha + \beta + \gamma - \pi) + \pi -\tau & =  \emm_1\alpha + \emm_2\pi/3 + \emm_3\pi/2. \label{e:rce_pro_tau}
\end{align}

Note that $2 \le n < 2\tau/\Delta(\T_0)$, since $\T \neq \T_0$ and the volume of the $\tau$-lune is $2\tau$.
From the solutions of the equation \eqref{e:rce_pro_tau} we
get all possible values of $\varphi + \psi$ and hence also all possible pairs $\varphi,\psi$.
We keep only those triples $\tau, \varphi,\psi$ that satisfy the spherical triangle inequality (Lemma~\ref{lemma_fact}(b)).
Using the spherical law of cosines (Lemma~\ref{lemma_fact}(c)), we compute the length of the edge $x$ of $\T$ opposite to the angle $\tau$ and check whether it can be obtained as a nonnegative integer combination of $a,b,c$. If not, then $\T$ is clearly not realizable. 

Since the calculations are tedious and it is relatively easy to make a mistake or forget some case, we decided to write a computer program to 
go through all the possibilities. The program was written in Sage 5.4.1~\cite{sage} and can be found in Appendix~\ref{s:appendix}.
The program takes five arguments $\alpha, \beta, \gamma, \tau,\rho$ and searches for all possible realizable $(\tau\varphi\psi)$-triangles 
with $\rho < \varphi \le \psi$ as follows. First, the program goes through all $4$-tuples $n, \emm_1, \emm_2, \emm_3$ satisfying equation~\eqref{e:rce_pro_tau}. Then
it tries to split $\emm_1\alpha + \emm_2\beta + \emm_3\gamma$ between $\varphi$ and $\psi$. If two splittings provide the same pair $(\varphi,\psi)$,
we list it just once. Then we test the condition (b) from Lemma~\ref{lemma_fact} for the $(\tau\varphi\psi)$-triangle.
As the last step, we use the spherical law of cosines (Lemma~\ref{lemma_fact}(c)) to approximate the length of the edge $x$ opposite to the angle $\tau$ and test whether it can be obtained as an approximate nonnegative integer combination of $a,b,c$. If not, we list the nonnegative integer combinations of $a,b,c$ that are closest to $x$. 
In the output we round the numerical values to three decimal places.

For every triangle $\T=(\tau\varphi\psi)$ not excluded by the program, we try to find some tiling of $\T$ by hand.
We use Observation~\ref{obs_dlazdeni_cocky} to simplify the search: in many cases, it will be enough to
find a tiling of a corner-filling triangle whose two copies fill the whole $\tau$-lune.

1) First we find realizable $(\alpha**)$-triangles for $\alpha = \pi/4$.
Table~\ref{t:pi_4_alfa} lists the output of the program for $\alpha, \beta, \gamma, \alpha, 0$. 
The four highlighted pairs $(\varphi,\psi)$ in Table~\ref{t:pi_4_alfa} are the only candidates for realizable $(\alpha\varphi\psi)$-triangles. Figure~\ref{fig_pi-4}, left, shows that all four candidates from Table~\ref{t:pi_4_alfa} are indeed realizable.
The realizability of the triangles $A'BD$ and $A'CB$ in Figure~\ref{fig_pi-4}, left, follows by Observation~\ref{obs_dlazdeni_cocky} since $A'DE$ is a copy of $ADE$.

\begin{table}[Htb!]
\begin{center}
\begin{tabu}{ccllc}
$n$ & $(\emm_1,\emm_2,\emm_3)$ & $(\varphi, \psi)$& $x$ & best approximation for $x$ using $a,b,c$\\
\hline
\noalign{\vskip 1mm}

2 & (1,2,0) & $(\beta,\alpha+\beta)$&0.808 &$0.785 \sim b < x < c \sim 0.955$\\
  & & $(\alpha,2\beta)$ {\cellcolor{zluta}} & 0.955 & $x = c$  \\ 
3 & (0,0,2) & $(\gamma,\gamma)$ {\cellcolor{zluta}} & 0.785 & $x =  b$ \\ 
  & (0,3,0) &$(\beta,2\beta)$&0.915 &$0.785 \sim b < x < c \sim 0.955$\\ 
  & (2,0,1) &$(\alpha,\alpha+\gamma)$&1.144 &$0.955 \sim c < x < 2a \sim 1.231$\\ 
4 & (1,1,1) &$(\gamma,\alpha+\beta)$&0.749 &$0.615 \sim a < x < b \sim 0.785$\\
  & & $(\beta,\alpha+\gamma)$ {\cellcolor{zluta}} & 0.955 & $x =  c$ \\
  & &$(\alpha,\beta+\gamma)$&1.300 &$1.231 \sim 2a < x < a+b \sim 1.401$\\ 
5 & (0,2,1) & $(\gamma,2\beta)$ {\cellcolor{zluta}} &  0.615 & $x =  a$ \\
  & &$(\beta,\beta+\gamma)$&0.885 &$0.785 \sim b < x < c \sim 0.95532$\\ 
  & (2,2,0) &$(\alpha,\alpha+2\beta)$&1.439 &$1.401 \sim a+b < x < 2b \sim 1.571$\\
  & & $(\alpha+\beta,\alpha+\beta)$&0.592 &$0.0 \sim 0 < x < a \sim 0.615$\\

\end{tabu}
\caption[$(\alpha**)$-triangles for $\alpha = \pi/4$.]{$(\alpha**)$-triangles for $\alpha = \pi/4$. Approximate values of $a,b,c$ are $a \sim 0.615$, $b\sim 0.785$ and $c \sim 0.955$.}\label{t:pi_4_alfa}
\end{center}
\end{table}

Now we find all realizable $(\beta \varphi \psi)$-triangles. 
 Since we have characterized all realizable $(\alpha **)$-triangles, we 
 can assume that $\varphi, \psi > \alpha$.
 Table~\ref{t:pi_4_beta} lists the output of the program for $\alpha, \beta, \gamma, \beta, \alpha$. Figure~\ref{fig_pi-4}, right, shows that all the candidates from Table~\ref{t:pi_4_alfa} are realizable. Again, triangles $B'AD$ and $B'CA$ in Figure~\ref{fig_pi-4}, right, are realizable by Observation~\ref{obs_dlazdeni_cocky} since $B'AG$ is a copy of $BGA$.

\begin{table}[Htb!]
\begin{center}
\begin{tabu}{ccllc}
$n$ & $(\emm_1,\emm_2,\emm_3)$ & $(\varphi, \psi)$& $x$ & best approximation for $x$ using $a,b,c$\\
\hline
\noalign{\vskip 1mm}

2 & (0,1,1) & $(\beta,\gamma)$ {\cellcolor{zluta}} & 0.955 & $x = c$ \\ 
3 & (1,2,0) & $(\beta,\alpha+\beta)$ & 1.112 & $0.955 \sim c < x < 2a \sim 1.231$ \\
4 & (0,0,2) & $(\gamma,\gamma)$ & 1.047 & $0.955 \sim c < x < 2a \sim 1.231$ \\   
  & (0,3,0) & $(\beta,2\beta)$ {\cellcolor{zluta}} & 1.231 & $x =  2a$ \\ 
5 & (1,1,1) & $(\gamma,\alpha+\beta)$ & 1.027 & $0.955 \sim c < x < 2a \sim 1.231$ \\
  & & $(\beta,\alpha+\gamma)$ & 1.329 & $1.231 \sim 2a < x < a+b \sim 1.401$ \\ 
6 & (0,2,1) & $(\gamma,2\beta)$ {\cellcolor{zluta}} & 0.955 & $x = c$ \\
  & & $(\beta,\beta+\gamma)$ & 1.415 & $1.401 \sim a+b < x < 2b \sim 1.571$ \\ 
  & (2,2,0) & $(\alpha+\beta,\alpha+\beta)$ & 0.918 & $0.785 \sim b < x < c \sim 0.955$ \\
7 & (1,0,2)   & $(\gamma,\alpha+\gamma)$ {\cellcolor{zluta}} & 0.785 & $x = b$ \\ 
  & (1,3,0) & $(\beta,\alpha+2\beta)$ & 1.495 & $1.401 \sim a+b < x < 2b \sim 1.571$ \\
  & & $(\alpha+\beta,2\beta)$ & 0.719 & $0.615 \sim a < x < b \sim 0.785$ \\

\end{tabu}
\caption[$(\beta \varphi \psi)$-triangles for $\alpha = \pi/4$.]{$(\beta \varphi \psi)$-triangles for $\alpha = \pi/4$ and $\varphi, \psi > \alpha$. Approximate values of $a,b,c$ are $a\sim 0.615$, $b \sim 0.785$ and $c\sim 0.955$.}\label{t:pi_4_beta}
\end{center}
\end{table}

\begin{figure}
 \begin{center}
  \begin{minipage}{.25\textwidth}
  \begin{tabular}{l|l}
  \multicolumn{2}{c}{$(\alpha\varphi\psi)$-triangles} \\ \hline
   $(\alpha,\alpha,2\beta)$ & $ ADB$\\
    $(\alpha,\gamma,\gamma) $ & $ ADE$\\
     $(\alpha,\beta,3\alpha) $ & $ A'BD$\\
      $(\alpha,\gamma,2\beta) $ & $ A'CB$
  \end{tabular}
  \end{minipage}
  \begin{minipage}{.35\textwidth}
   \includegraphics[scale=.9]{obrazky/sphere-troj.4}
    \includegraphics[scale=.9]{obrazky/sphere-troj.5}
  \end{minipage}
   \begin{minipage}{.25\textwidth}
   \begin{tabular}{l|l}
  \multicolumn{2}{c}{$(\beta\varphi\psi)$-triangles} \\ \hline
   $(\beta,\beta,\gamma) $&$ BDA$\\
    $(\beta,\beta,2\beta) $&$ BFD$\\
     $(\beta,\gamma,2\beta) $&$  B'AD$\\
      $(\beta,\gamma,3\alpha) $&$ B'CA$
  \end{tabular}
      \end{minipage}
      \bigskip
      
      \includegraphics{obrazky/popisky.2}
   \caption{A tiling of $(\alpha**)$-triangles (left) and $(\beta**)$-triangles (right) for $\alpha = \frac{\pi}{4}$.}
\label{fig_pi-4}
 \end{center}
\end{figure}
 
2) First we find realizable $(\alpha **)$-triangles for $\alpha = \pi/5$. 
Table~\ref{t:pi_5_alfa} lists the output of the program for $\alpha, \beta, \gamma, \alpha, 0$.
Figure~\ref{fig_pi-5}, left, shows that all the candidates from Table~\ref{t:pi_5_alfa} are realizable. The triangles $A'DF, A'ED, A'BD, A'CB$ are realizable by Observation~\ref{obs_dlazdeni_cocky} since $A'HD$ is a copy of $ADH$.

\begin{table}[Htb!]
\begin{center}
\begin{tabu}{ccllc}
$n$ & $(\emm_1,\emm_2,\emm_3)$ & $(\varphi, \psi)$& $x$ & best approximation for $x$ using $a,b,c$\\
\hline
\noalign{\vskip 1mm}

2 & (1,2,0) & $(\beta,\alpha+\beta)$ & 0.498 & $0.365 \sim a < x < b \sim 0.554$ \\
  & & $(\alpha,2\beta)$ {\cellcolor{zluta}} & 0.652 & $x = c$ \\ 
3 & (2,0,1) & $(\alpha,\alpha+\gamma)$&0.794 &$0.730 \sim 2a < x < a+b \sim 0.918$ \\
  & & $(2\alpha,\gamma)$ {\cellcolor{zluta}} & 0.554 & $x = b$ \\ 
4 & (3,1,0) & $(\beta,3\alpha)$ {\cellcolor{zluta}} & 0.652 & $x = c$ \\
  & & $(\alpha,2\alpha+\beta)$ & 0.911 & $0.730 \sim 2a < x < a+b \sim 0.918$ \\
  & & $(2\alpha,\alpha+\beta)$ & 0.607 & $0.554 \sim b < x < c \sim 0.652$ \\ 
6 & (0,0,2) & $(\gamma,\gamma)$ & 0.628 &$0.554 \sim b < x < c \sim 0.652$ \\ 
  & (0,3,0) & $(\beta,2\beta)$ {\cellcolor{zluta}} & 0.730 & $x = 2a$ \\ 
  & (5,0,0) & $(\alpha,4\alpha)$ {\cellcolor{zluta}} & 1.107 & $x = 2b$ \\
  & & $(2\alpha,3\alpha)$ & 0.662 & $0.652 \sim c < x < 2a \sim 0.730$ \\ 
7 & (1,1,1) & $(\gamma,\alpha+\beta)$ & 0.620 & $0.554 \sim b < x < c \sim 0.652$ \\
  & & $(\beta,\alpha+\gamma)$ & 0.745 & $0.730 \sim 2a < x < a+b \sim 0.918$ \\
  & & $(\alpha,\beta+\gamma)$ & 1.193 & $1.107 \sim 2b < x < b+c \sim 1.206$ \\ 
8 & (2,2,0) & $(\beta,2\alpha+\beta)$ & 0.742 & $0.730 \sim 2a < x < a+b \sim 0.918$ \\
  & & $(\alpha,\alpha+2\beta)$ & 1.274 & $1.206 \sim b+c < x < 2a+b \sim 1.283$ \\
  & & $(\alpha+\beta,\alpha+\beta)$ & 0.593 & $0.554 \sim b < x < c \sim 0.652$ \\
  & & $(2\alpha,2\beta)$ {\cellcolor{zluta}} & 0.652 & $x = c$ \\
9 & (3,0,1) & $(\gamma,3\alpha)$ {\cellcolor{zluta}} & 0.554 & $x = b$ \\
  & & $(\alpha,2\alpha+\gamma)$ & 1.351 & $1.305 \sim 2c < x < 2a+c \sim 1.382$ \\
  & & $(2\alpha,\alpha+\gamma)$ & 0.617 & $0.554 \sim b < x < c \sim 0.652$ \\ 
10 & (4,1,0) & $(\beta,4\alpha)$ {\cellcolor{zluta}} & 0.652 & $x = c$ \\
  & & $(\alpha,3\alpha+\beta)$ & 1.426 & $1.382 \sim 2a+c < x < 4a \sim 1.459$ \\
  & & $(\alpha+\beta,3\alpha)$ & 0.475 & $0.365 \sim a < x < b \sim 0.554$ \\
  & & $(2\alpha,2\alpha+\beta)$ & 0.551 & $0.365 \sim a < x < b \sim 0.554$ \\ 
11 & (0,2,1) & $(\gamma,2\beta)$ {\cellcolor{zluta}} & 0.365 & $x = a$ \\
  & & $(\beta,\beta+\gamma)$ & 0.519 & $0.365 \sim a < x < b \sim 0.554$ \\

\end{tabu}
\caption[$(\alpha **)$-triangles for $\alpha = \pi/5$.]{$(\alpha **)$-triangles for $\alpha = \pi/5$. Approximate values of $a,b,c$ are $a\sim0.365$, $b\sim0.554$ and $c\sim0.652$.}\label{t:pi_5_alfa}
\end{center}
\end{table}

Now we find all realizable $(\beta \varphi \psi)$-triangles. We may again assume that $\varphi, \psi > \alpha$.
Table~\ref{t:pi_5_beta} lists the output of the program for $\alpha, \beta, \gamma, \beta, \alpha$. Figure~\ref{fig_pi-5}, right, shows that all the candidates from Table~\ref{t:pi_5_beta} are realizable.
The triangles $B'GI, B'DH, B'AI, B'AD$ and $B'CA$ are realizable by Observation~\ref{obs_dlazdeni_cocky} since $B'IH$ is a copy of $BHI$.

\begin{table}[Htb!]
\begin{center}
\begin{tabu}{ccllc}
$n$ & $(\emm_1,\emm_2,\emm_3)$ & $(\varphi, \psi)$& $x$ & best approximation for $x$ using $a,b,c$\\
\hline
\noalign{\vskip 1mm}

2 & (2,1,0) & $(\beta,2\alpha)$ {\cellcolor{zluta}} & 0.652 & $x = c$ \\ 
4 & (4,0,0) & $(2\alpha,2\alpha)$ & 0.852 & $0.730 \sim 2a < x < a+b \sim 0.918$ \\ 
5 & (0,1,1) & $(\beta,\gamma)$ & 0.955 & $0.918 \sim a+b < x < a+c \sim 1.017$ \\ 
6 & (1,2,0) & $(\beta,\alpha+\beta)$ & 1.024 & $1.017 \sim a+c < x < 3a \sim 1.095$\\ 
7 & (2,0,1) & $(2\alpha,\gamma)$ {\cellcolor{zluta}} & 1.017 &  $x =  a+c$ \\ 
8 & (3,1,0) & $(\beta,3\alpha)$ & 1.138 & $1.107 \sim 2b < x < b+c \sim 1.206$\\
  & & $(2\alpha,\alpha+\beta)$ & 1.054 & $1.017 \sim a+c < x < 3a \sim 1.095$\\ 
10 & (0,0,2) & $(\gamma,\gamma)$ & 1.047 & $1.017 \sim a+c < x < 3a \sim 1.095$\\
   & (0,3,0) & $(\beta,2\beta)$ & 1.231 & $1.206 \sim b+c < x < 2a+b \sim 1.283$\\ 
   & (5,0,0) & $(2\alpha,3\alpha)$ {\cellcolor{zluta}} & 1.107 & $x =  2b$ \\ 
11 & (1,1,1) & $(\gamma,\alpha+\beta)$ & 1.044 & $1.017 \sim a+c < x < 3a \sim 1.095$\\
   & & $(\beta,\alpha+\gamma)$ & 1.272 & $1.206 \sim b+c < x < 2a+b \sim 1.283$\\ 
12 & (2,2,0) & $(\beta,2\alpha+\beta)$ & 1.311 & $1.305 \sim 2c < x < 2a+c \sim 1.382$\\
   & & $(\alpha+\beta,\alpha+\beta)$ & 1.028 & $1.017 \sim a+c < x < 3a \sim 1.095$\\
   & & $(2\alpha,2\beta)$ & 1.138 & $ 1.107 \sim 2b < x < b+c \sim 1.206$\\ 
13 & (3,0,1) & $(\gamma,3\alpha)$ {\cellcolor{zluta}} & 1.017 &  $x = a+c$ \\
   & & $(2\alpha,\alpha+\gamma)$ & 1.144 & $1.107 \sim 2b < x < b+c \sim 1.206$\\ 
14 & (4,1,0) & $(\beta,4\alpha)$ {\cellcolor{zluta}} & 1.382 & $x = 2a+c$ \\
   & & $(\alpha+\beta,3\alpha)$ & 0.973 & $0.918 \sim a+b < x < a+c \sim 1.017$\\
   & & $(2\alpha,2\alpha+\beta)$ & 1.143 & $1.107 \sim 2b < x < b+c \sim 1.206$\\ 
15 & (0,2,1) & $(\gamma,2\beta)$ & 0.955 & $0.918 \sim a+b < x < a+c \sim 1.017$\\
   & & $(\beta,\beta+\gamma)$ & 1.415 & $1.382 \sim 2a+c < x < 4a \sim 1.459$\\ 
16 & (1,0,2) & $(\gamma,\alpha+\gamma)$ & 0.905 & $0.730 \sim 2a < x < a+b \sim 0.918$\\ 
   & (1,3,0) & $(\beta,\alpha+2\beta)$ & 1.448 & $1.382 \sim 2a+c < x < 4a \sim 1.459$\\
   & & $(\alpha+\beta,2\beta)$ & 0.875 & $0.730 \sim 2a < x < a+b \sim 0.918$\\ 
   & (6,0,0) & $(2\alpha,4\alpha)$ {\cellcolor{zluta}} & 1.107 &  $x = 2b$ \\
   & & $(3\alpha,3\alpha)$ & 0.852 & $0.730 \sim 2a < x < a+b \sim 0.918$\\ 
17 & (2,1,1) & $(\gamma,2\alpha+\beta)$ & 0.833 & $0.730 \sim 2a < x < a+b \sim 0.918$\\
   & & $(\beta,2\alpha+\gamma)$ & 1.479 & $1.472 \sim a+2b < x < a+b+c \sim 1.571$\\
   & & $(\alpha+\beta,\alpha+\gamma)$ & 0.798 & $0.730 \sim 2a < x < a+b \sim 0.918$\\
   & & $(2\alpha,\beta+\gamma)$ & 1.060 & $1.017 \sim a+c < x < 3a \sim 1.095$\\ 
18 & (3,2,0) & $(\beta,3\alpha+\beta)$ & 1.510 & $1.472 \sim a+2b < x < a+b+c \sim 1.571$\\
   & & $(\alpha+\beta,2\alpha+\beta)$ & 0.690 & $0.652 \sim c < x < 2a \sim 0.730$\\
   & & $(2\alpha,\alpha+2\beta)$ & 0.973 & $0.918 \sim a+b < x < a+c \sim 1.017$\\
   & & $(3\alpha,2\beta)$ {\cellcolor{zluta}} & 0.652 & $x =  c$ \\ 
19 & (4,0,1) & $(\gamma,4\alpha)$ {\cellcolor{zluta}} & 0.554 & $x = b$ \\
   & & $(2\alpha,2\alpha+\gamma)$ & 0.794 & $0.730 \sim 2a < x < a+b \sim 0.918$\\
   & & $(3\alpha,\alpha+\gamma)$ & 0.482 & $0.365 \sim a < x < b \sim 0.554$\\

\end{tabu}
\caption[$(\beta \varphi \psi)$-triangles for $\alpha = \pi/5$.]{$(\beta \varphi \psi)$-triangles for $\alpha = \pi/5$ and $\varphi, \psi > \alpha$. Approximate values of $a,b,c$ are $a\sim0.365$, $b\sim0.554$ and $c\sim0.652$.}\label{t:pi_5_beta}
\end{center}
\end{table}

\begin{figure}
\begin{center}  
  \begin{minipage}{.27\textwidth}
  \begin{tabular}{l|l}
  \multicolumn{2}{c}{$(\alpha\varphi\psi)$-triangles} \\ \hline
   $(\alpha,\alpha,2\beta)$ & $ ADB$\\
    $(\alpha,2\alpha,\gamma) $ & $ ADE$\\
     $(\alpha,\beta,3\alpha) $ & $ AFD$\\
      $(\alpha,\beta,2\beta) $ & $ AJF$\\
       $(\alpha,\alpha,4\alpha) $ & $ AHD$\\
     $(\alpha,2\alpha,2\beta) $ & $ A'DF$\\
      $(\alpha,\gamma,3\alpha) $ & $ A'ED$\\
     $(\alpha,\beta,4\alpha) $ & $ A'BD$\\
       $(\alpha,\gamma,2\beta) $ & $ A'CB$
  \end{tabular}
  \end{minipage}
  \begin{minipage}{.4\textwidth}
   \includegraphics[scale=.9]{obrazky/sphere-troj.7}
    \includegraphics[scale=.9]{obrazky/sphere-troj.9}
  \end{minipage}
   \begin{minipage}{.3\textwidth}
   \begin{tabular}{l|l}
  \multicolumn{2}{c}{$(\beta\varphi\psi)$-triangles} \\ \hline
   $(\beta,\beta,2\alpha) $&$ BDA$\\
    $(\beta,2\alpha,\gamma) $&$ BIG$\\
     $(\beta,2\alpha,3\alpha) $&$  BHI$\\
     $(\beta,\gamma,3\alpha) $&$  B'GI$\\
     $(\beta,\beta,4\alpha) $&$  B'DH$\\
      $(\beta,2\alpha,4\alpha) $&$  B'AI$\\
     $(\beta,3\alpha,2\beta) $&$  B'AD$\\
      $(\beta,\gamma,4\alpha) $&$ B'CA$
  \end{tabular}
      \end{minipage}
      \bigskip
    
      \includegraphics{obrazky/popisky.2}
  \caption{A tiling of $(\alpha**)$-triangles (left) and $(\beta**)$-triangles (right) for $\alpha = \frac{\pi}{5}$.}
\label{fig_pi-5}
 \end{center}
\end{figure}

3) First we find all realizable $(\alpha**)$-triangles. 
Table~\ref{t:2pi_9_alfa} lists the output of the program for $\alpha, \beta, \gamma, \alpha, 0$. 
Figure~\ref{fig_2pi-9}, left, shows that all the candidates from Table~\ref{t:2pi_9_alfa} are realizable. The triangles $A'ED, A'BD$ and $A'CB$ are realizable by
Observation~\ref{obs_dlazdeni_cocky} since $A'DF$ is a copy of $AFD$.

\begin{table}[Htb!]
\begin{center}
\begin{tabu}{ccllc}
$n$ & $(\emm_1,\emm_2,\emm_3)$ & $(\varphi, \psi)$& $x$ & best approximation for $x$ using $a,b,c$\\
\hline
\noalign{\vskip 1mm}

2 & (1,2,0) & $(\beta,\alpha+\beta)$ & 0.649 & $0.485 \sim a < x < b \sim 0.680$ \\
  & & $(\alpha,2\beta)$ {\cellcolor{zluta}} & 0.812 & $x = c$ \\ 
  & (4,0,0) & $(2\alpha,2\alpha)$ & 0.608 & $ 0.485 \sim a < x < b \sim 0.680$ \\ 
3 & (2,0,1) & $(\alpha,\alpha+\gamma)$ & 0.982 & $0.971 \sim 2a < x < a+b \sim 1.165$ \\
  & & $(2\alpha,\gamma)$ {\cellcolor{zluta}} & 0.680 & $x =  b$ \\ 
4 & (0,0,2) & $(\gamma,\gamma)$ & 0.698 & $0.680 \sim b < x < c \sim 0.812$ \\ 
  & (0,3,0) & $(\beta,2\beta)$ {\cellcolor{zluta}} & 0.812 & $x =  c$ \\ 
  & (3,1,0) & $(\alpha,2\alpha+\beta)$ &1.122 & $ 0.971 \sim 2a < x < a+b \sim 1.165$ \\
  & & $(2\alpha,\alpha+\beta)$ & 0.709 & $0.680 \sim b < x < c \sim 0.812$ \\ 
5 & (1,1,1) & $(\gamma,\alpha+\beta)$ {\cellcolor{zluta}} & 0.680 & $x = b$ \\
  & & $(\beta,\alpha+\gamma)$ & 0.836 & $0.812 \sim c < x < 2a \sim 0.971$ \\
  & & $(\alpha,\beta+\gamma)$ & 1.246 & $1.164 \sim a+b < x < a+c \sim 1.297$ \\ 
6 & (2,2,0) & $(\beta,2\alpha+\beta)$ {\cellcolor{zluta}} & 0.812 & $x = c$ \\
  & & $(\alpha,\alpha+2\beta)$ {\cellcolor{zluta}} & 1.359 & $x = 2b$ \\
  & & $(\alpha+\beta,\alpha+\beta)$ &0.608 & $0.485 \sim a < x < b \sim 0.680$ \\
  & & $(2\alpha,2\beta)$ &0.649 & $0.485 \sim a < x < b \sim 0.680$ \\ 
7 & (0,2,1) & $(\gamma,2\beta)$ {\cellcolor{zluta}} & 0.485 & $x = a$ \\
  & & $(\beta,\beta+\gamma)$ &0.693 & $0.680 \sim b < x < c \sim 0.812$ \\
  & (3,0,1) & $(\alpha,2\alpha+\gamma)$ &1.466 & $1.456 \sim 3a < x < b+c \sim 1.492$ \\
  & & $(2\alpha,\alpha+\gamma)$ &0.521 & $0.485 \sim a < x < b \sim 0.680$ \\

\end{tabu}
\caption[$(\alpha **)$-triangles for $\alpha = 2\pi/9$.]{$(\alpha **)$-triangles for $\alpha = 2\pi/9$. Approximate values of $a,b,c$ are $a\sim0.485$, $b\sim0.680$ and $c\sim0.812$.}\label{t:2pi_9_alfa}
\end{center}
\end{table}


It remains to find all realizable $(\beta **)$-triangles. Table~\ref{t:2pi_9_beta} lists the output of the program for $\alpha, \beta, \gamma, \beta, \alpha$.
We show that the $(\beta,\beta,2\alpha+\beta)$-triangle is not realizable. This might be a bit surprising, as the edges of this triangle have lengths $2b,2b$ and $2a+2b$. Suppose that $\T=(\beta,\beta,2\alpha+\beta)$ is realizable. Let $B,U,V$ be the vertices of $\T$ so that $BU$ and $UV$ are edges of length $2b$.
We may assume that $\T$ is corner-filling in the $\beta$-lune with vertex $B$.
Since $\T$ is realizable and $\beta$ is not a multiple of $\alpha$, we may assume that in a tiling of $\T$ by copies of $\T_0$, some tile is also a corner-filling tile of the $\beta$-lune.
It follows that the induced tiling of the edge $BU$ of length $2b$ by edges $a,b,c$ contains at least one edge $a$ or $c$.
However, it is easy to check that neither $2b-a\sim 0.874$ nor $2b-c\sim 0.547$ can be expressed as a nonnegative integer combination of $a,b,c$.

\begin{table}[Htb!]
\begin{center}
\begin{tabu}{ccllc}
$n$ & $(\emm_1,\emm_2,\emm_3)$ & $(\varphi, \psi)$& $x$ & best approximation for $x$ using $a,b,c$\\
\hline
\noalign{\vskip 1mm}

2 & (2,1,0) & $(\beta,2\alpha)$ {\cellcolor{zluta}} & 0.812 & $x = c$ \\ 
3 & (0,1,1) & $(\beta,\gamma)$ & 0.955 & $0.812 \sim c < x < 2a \sim 0.971$ \\ 
4 & (1,2,0) & $(\beta,\alpha+\beta)$ & 1.065 & $0.971 \sim 2a < x < a+b \sim 1.165$ \\ 
  & (4,0,0) & $(2\alpha,2\alpha)$ & 0.992 & $0.971 \sim 2a < x < a+b \sim 1.165$ \\ 
5 & (2,0,1) & $(2\alpha,\gamma)$ & 1.038 & $0.971 \sim 2a < x < a+b \sim 1.165$ \\ 
6 & (0,0,2) & $(\gamma,\gamma)$ & 1.047 & $0.971 \sim 2a < x < a+b \sim 1.165$ \\ 
  & (0,3,0) & $(\beta,2\beta)$ & 1.231 & $1.165 \sim a+b < x < a+c \sim 1.297$ \\ 
  & (3,1,0) & $(2\alpha,\alpha+\beta)$ & 1.065 & $0.971 \sim 2a < x < a+b \sim 1.165$ \\ 
7 & (1,1,1) & $(\gamma,\alpha+\beta)$ & 1.038 & $0.971 \sim 2a < x < a+b \sim 1.165$ \\
  & & $(\beta,\alpha+\gamma)$ & 1.298 & $1.297 \sim a+c < x < 2b \sim 1.359$ \\ 
8 & (2,2,0) & $(\beta,2\alpha+\beta)$ {\cellcolor{zluta}} & 1.359 & $x =  2b$ \\
  & & $(\alpha+\beta,\alpha+\beta)$ & 0.992 & $0.971 \sim 2a < x < a+b \sim 1.165$ \\
  & & $(2\alpha,2\beta)$ & 1.065 & $0.971 \sim 2a < x < a+b \sim 1.165$ \\ 
9& (0,2,1) & $(\gamma,2\beta)$ & 0.955 & $0.812 \sim c < x < 2a \sim 0.971$ \\
  & & $(\beta,\beta+\gamma)$ & 1.415 & $1.359 \sim 2b < x < 3a \sim 1.456$ \\ 
  & (3,0,1) & $(2\alpha,\alpha+\gamma)$ & 1.030 & $0.971 \sim 2a < x < a+b \sim 1.165$ \\ 
10& (1,0,2) & $(\gamma,\alpha+\gamma)$ & 0.860 & $0.812 \sim c < x < 2a \sim 0.971$ \\ 
  & (1,3,0) & $(\beta,\alpha+2\beta)$ & 1.469 & $1.456 \sim 3a < x < b+c \sim 1.492$ \\
  & & $(\alpha+\beta,2\beta)$ {\cellcolor{zluta}} & 0.812 & $x =  c$ \\ 
  & (4,1,0) & $(2\alpha,2\alpha+\beta)$ & 0.952 & $0.812 \sim c < x < 2a \sim 0.971$ \\ 
11& (2,1,1) & $(\gamma,2\alpha+\beta)$ {\cellcolor{zluta}} & 0.680 & $x = b$ \\
  & & $(\beta,2\alpha+\gamma)$ & 1.520 & $1.492 \sim b+c < x < 2c \sim 1.624$ \\
  & & $(\alpha+\beta,\alpha+\gamma)$ & 0.625 & $0.485 \sim a < x < b \sim 0.680$ \\
  & & $(2\alpha,\beta+\gamma)$ & 0.781 & $0.680 \sim b < x < c \sim 0.812$ \\

\end{tabu}
\caption[$(\beta \varphi \psi)$-triangles for $\alpha = 2\pi/9$.]{$(\beta \varphi \psi)$-triangles for $\alpha = 2\pi/9$ and $\varphi,\psi > \alpha$. Approximate values of $a,b,c$ are $a\sim0.485$, $b\sim0.680$ and $c\sim0.812$.}\label{t:2pi_9_beta}
\end{center}
\end{table}

Figure~\ref{fig_2pi-9}, right, shows that all the other candidates are realizable. Triangles $B'AD$ and $B'CA$ are realizable by Observation
\ref{obs_dlazdeni_cocky} since $B'HA$ is a copy of $BAH$.

This finishes the proof of Lemma~\ref{l_pi-4-5-9}.
\end{proof}

\begin{figure}
 \begin{center}
  \begin{minipage}{.29\textwidth}
  \begin{tabular}{l|l}
  \multicolumn{2}{c}{$(\alpha\varphi\psi)$-triangles} \\ \hline
   $(\alpha,\alpha,2\beta)$ & $ ADB$\\
    $(\alpha,2\alpha,\gamma) $ & $ ADE$\\
     $(\alpha,\beta,2\beta) $ & $ AFD$\\
      $(\alpha,\gamma,\alpha+\beta) $ & $ A'ED$\\
      $(\alpha,\beta,2\alpha+\beta) $ & $ A'BD$\\
      $(\alpha,\alpha,4\alpha) $ & $ AHD$ \\
      $(\alpha,\gamma,2\beta) $ & $ A'CB$
  \end{tabular}
  \end{minipage}
  \hspace{.4cm}
  \begin{minipage}{.4\textwidth}
   \includegraphics[scale=.9]{obrazky/sphere-troj.10} 
   \hspace{.4cm}
    \includegraphics[scale=.9]{obrazky/sphere-troj.12}
  \end{minipage}
  \hspace{-1cm}
   \begin{minipage}{.30\textwidth}
    \begin{tabular}{l|l}
  \multicolumn{2}{c}{$(\beta\varphi\psi)$-triangles} \\ \hline
   $(\beta,\beta,2\alpha)$ & $ BDA$\\
    $(\beta,\alpha+\beta,2\beta) $ & $ B'AD$\\
      $(\beta,\gamma,2\alpha+\beta) $ & $ B'CA$
  \end{tabular}
      \end{minipage}
      \bigskip
      
      \includegraphics{obrazky/popisky.2} 
  \caption{A tiling of $(\alpha**)$-triangles (left) and $(\beta**)$-triangles (right) for $\alpha = \frac{2\pi}{9}$.}
\label{fig_2pi-9}
 \end{center}
\end{figure}

The following lemma extends Lemma~\ref{lem_partial_classification_cox_2_3_alpha_edges}.

\begin{lemma}\label{lem_alfa_beta_diagramy}
 Let $\T_0 = (\alpha\beta\gamma)$, where $\alpha < \beta < \gamma$ and $\beta = \pi/3$. Then the $\alpha$-edges and $\beta$-edges of $S$ together form a subgraph isomorphic to one of the six
 graphs in Figure~\ref{fig_alfa_beta_diagramy}.
\end{lemma}

\begin{figure}
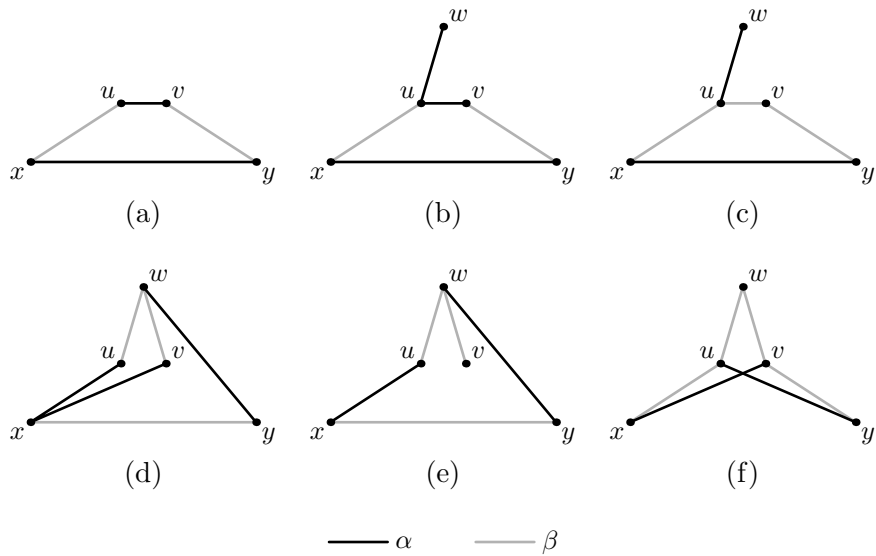

 \begin{center}
 \begin{tabular}{ccc}
    \includegraphics{obrazky/coxeter.100} & \includegraphics{obrazky/coxeter.101} & \includegraphics{obrazky/coxeter.102} \\
    (a) & (b) & (c) \\
    & & \\
    \includegraphics{obrazky/coxeter.103} & \includegraphics{obrazky/coxeter.104} & \includegraphics{obrazky/coxeter.105} \\
     (d) & (e) & (f) \\
       & & \\
     &  \includegraphics{obrazky/popisky.200} & \\
 \end{tabular}
  \caption{$\alpha$-edges and $\beta$-edges of $S$ for $\beta = \pi/3$.}
\label{fig_alfa_beta_diagramy}
 \end{center}
\end{figure}

\begin{proof}
Let $H=c(S)$ be the Coxeter diagram of $S$. Let $V(H)=\{u,v,w,x,y\}$. Let $H_{\alpha}=(V(H),E_{\alpha})$ and $H_{\beta}=(V(H),E_{\beta})$ be the subgraphs of $H$ formed by the $\alpha$-edges and $\beta$-edges, respectively. Let $H_{\alpha\beta}$
be the edge-labeled subgraph of $H$ formed by the $\alpha$-edges and the $\beta$-edges. Since $H$ is $(\alpha\beta\gamma)$-rich, we have $\lvert E_{\beta}\rvert \ge 2$. Furthermore, $H_{\alpha\beta}$ contains at least four induced \emph{$\alpha\beta$-paths}, that is, induced paths of length $2$ consisting of an $\alpha$-edge and a $\beta$-edge. Moreover, $H_{\alpha\beta}$ is triangle-free: indeed, by Lemma~\ref{lemma_fact}(a), there are no spherical triangles of type $(\alpha\alpha\alpha)$, $(\alpha\alpha\beta)$, $(\alpha\beta\beta)$ or $(\beta\beta\beta)$ for $\alpha<\beta=\pi/3$.
By Lemma~\ref{lem_partial_classification_cox_2_3_alpha_edges}, $H_{\alpha}$ is isomorphic to  $P_2 + P_2$ or $P_2 + P_3$.  

Suppose that $|E_{\beta}|= 2$. Then $H_{\beta}$ is a matching by the same argument as in the case $|E_{\alpha}| = 2$ in the proof of Lemma~\ref{lem_partial_classification_cox_2_3_alpha_edges}. Let $E_{\beta}=\{xu,yv\}$. If $H_{\alpha}$ also forms a matching, then to get four $\alpha\beta$-paths, $H_{\alpha\beta}$ must be isomorphic to an alternating $4$-cycle; see Figure~\ref{fig_alfa_beta_diagramy}(a). If $H_{\alpha}$ is isomorphic to $P_2 + P_3$, there are three possible non-isomorphic graphs $H_{\alpha\beta}$, determined by the choice of $E_{\alpha}$:
\begin{enumerate}
\item $E_{\alpha}=\{xy,yu,vw\}$. In this case there are only three induced $\alpha\beta$-paths, all containing the $\beta$-edge $yv$.
\item $E_{\alpha}=\{xy,uw,vw\}$. In this case there are exactly four $\alpha\beta$-paths, forcing $xw, yw, xv, yu$ to be $\gamma$-edges. Due to the symmetry 
simultaneously exchanging $u$ with $v$ and $x$ with $y$, the $(\alpha\beta\gamma)$-triangles form only two orbits; a contradiction. 
\item $E_{\alpha}=\{xy,uv,uw\}$. The graph $H_{\alpha\beta}$ is drawn in Figure~\ref{fig_alfa_beta_diagramy}(b).
\end{enumerate}

Suppose that $|E_{\beta}|= 3$. If $H_{\beta}$ is isomorphic to the star $K_{1,3}$, then there are at most three $(\alpha\beta\gamma)$-triangles. Also, $H_{\beta}$ cannot form a triangle as $H_{\alpha\beta}$ is triangle-free.

Suppose that $H_{\beta}$ is isomorphic to the path $P_4$. Let $E_{\beta}=\{xu,uv,vy\}$. There is no triangle-free extension by $\alpha$-edges where $H_{\alpha}$ is isomorphic to $P_2 + P_3$, and exactly one, up to isomorphism, where $H_{\alpha}$ is a matching; see Figure~\ref{fig_alfa_beta_diagramy}(c).

Suppose that $H_{\beta}$ is isomorphic to $P_2+P_3$. Let $E_\beta=\{uw,vw,xy\}$. Up to isomorphism, there are three triangle-free extensions of $H_{\beta}$ by $\alpha$-edges: $E_{\alpha}=\{xu,yv\}$, $E_{\alpha}=\{xu,yw\}$ and $E_{\alpha}=\{xu,xv,yw\}$. If $E_\alpha = \{xu,yv\}$, there are exactly four $\alpha\beta$-paths, forcing $xw, yw, xv, yu$ to be $\gamma$-edges. Due to the symmetry 
simultaneously exchanging $u$ with $v$ and $x$ with $y$, the $(\alpha\beta\gamma)$-triangles form only two orbits; a contradiction. The other two cases are displayed in Figure~\ref{fig_alfa_beta_diagramy}(d), (e).

Suppose that $|E_{\beta}| \ge 4$. Then by Tur\'an's theorem, $|E_{\beta}|= 4$, $|E_{\alpha}| = 2$ and $H_{\alpha\beta}$ is isomorphic to $K_{2,3}$. Since $H_{\alpha}$ is a matching, there is just one possibility for $H_{\alpha\beta}$ up to isomorphism; see Figure~\ref{fig_alfa_beta_diagramy}(f).
\end{proof}

Now we finish the description of all possible Coxeter diagrams of $S$.

\begin{lemma}\label{lem_abc_rich_pi4_pi_5_2pi_9}
Let $\gamma=\pi/2$ and $\beta=\pi/3$. For $\alpha = \pi/4$ and $\alpha=\pi/5$, Figures~\ref{fig_cox-pi-4} and~\ref{fig_cox-pi-5}, respectively, show all $(\alpha\beta\gamma)$-rich Coxeter diagrams with five vertices whose all triangles of type $(\alpha**)$ and $(\beta**)$ are listed in Lemma~\ref{l_pi-4-5-9}. For $\alpha = 2\pi/9$, no such diagram exists.
\end{lemma}

\begin{proof}
 Let $H$ be a Coxeter diagram satisfying the assumptions of the lemma. Let $V(H)=\{u,\allowbreak v,\allowbreak w,\allowbreak x,\allowbreak y\}$. Let $E_{\alpha}$ be the set of $\alpha$-edges and $E_{\beta}$ the set of $\beta$-edges of $H$. For each of the three values of $\alpha$, we consider the six cases of diagrams given by Lemma~\ref{lem_alfa_beta_diagramy}. We refer to these cases by (a)--(f) according to Figure~\ref{fig_alfa_beta_diagramy}.

First we exclude case (f) for all three values of $\alpha$. Suppose that $E_\alpha = \{xv,yu\}$ and $E_\beta=\{xu,uw,wv,vy\}$. Since $H$ is $(\alpha\beta\gamma)$-rich, the edges $xy$ and $uv$ must be $\gamma$-edges. This gives a $(\beta\beta\gamma)$-triangle $uvw$, which is realizable only for $\alpha=\pi/4$. The remaining edges $xw$ and $yw$ are also $\gamma$-edges since $\varphi=\gamma$ is the only angle such that the $(\beta\beta\varphi)$-triangle and the $(\alpha\beta\varphi)$-triangle are both realizable for $\alpha=\pi/4$. But then due to the symmetry exchanging simultaneously $u$ with $v$ and $x$ with $y$, the $(\alpha\beta\gamma)$-triangles form at most three orbits; a contradiction.

Now we continue separately for each value of $\alpha$.
 
 1) $\alpha = \pi/4$. Recall that by the spherical law of cosines (Lemma~\ref{lemma_fact}(c)), we have, approximately, $a\sim 0.615$, $b=\pi/4\sim 0.785$ and $c\sim 0.955$. The triangle of type $(\pi/2,\pi/2,2\beta)$ is not realizable since its longest edge has length $2\pi/3 \sim 2.094$, which cannot be obtained as a nonnegative integer combination of $a,b,c$. 
\begin{figure}
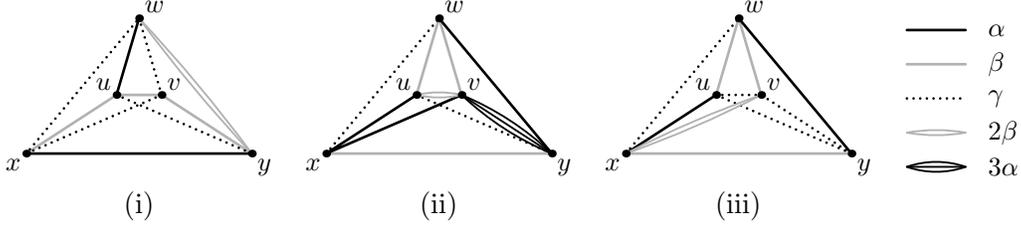

 \begin{center}
  \begin{tabular}{cccc} 
     \includegraphics{obrazky/coxeter.400} &   \includegraphics{obrazky/coxeter.401} &    \includegraphics{obrazky/coxeter.402} & \includegraphics{obrazky/popisky.110}\\
    (i) & (ii) & (iii) & \\
  \end{tabular}
  \caption{Coxeter diagrams for $\alpha = \frac{\pi}{4}$.}
\label{fig_cox-pi-4}
 \end{center}
\end{figure}

(a) $E_\alpha = \{uv,xy\}$, $E_\beta = \{xu,yv\}$. Since $H$ is $(\alpha\beta\gamma)$-rich, the diagonals $xv$ and $uy$ are $\gamma$-edges. Since $(\alpha\gamma\gamma)$ and $(\alpha,\gamma,2\beta)$ are the only realizable $(\alpha\varphi\psi)$-triangles for $\varphi,\psi \notin \{\alpha,\beta\}$, it follows that all edges incident with $w$ are of type $\gamma$ or $2\beta$. Since the $(\beta\gamma\gamma)$-triangle is not realizable, one of the edges $uw$, $xw$ is of type $2\beta$. We can assume without loss of generality that $uw$ is of type $2\beta$. Then, necessarily, $xw$ and $vw$ are of type $\gamma$ and $yw$ is of type $2\beta$.
The resulting diagram has a symmetry group generated by the transpositions $(x,v)$ and $(u,y)$, so it has at most two orbits of $(\alpha\beta\gamma)$-triangles. 

(b) $E_\alpha = \{uv,uw,xy\}$, $E_\beta = \{xu,yv\}$. Like in case (a), the edges $xv$ and $uy$ are of type $\gamma$. The types of the edges $vw$ and $yw$ are now uniquely determined: $vw$ has type $2\beta$ and $yw$ has type $\gamma$. If $\varphi$ is the type of $xw$ then the $(\alpha\beta\varphi)$-triangle and the $(\alpha\gamma\varphi)$-triangle are realizable, hence $\varphi=\gamma$. But then the triangle $xvw$ is a $(\gamma,\gamma,2\beta)$-triangle, which is not realizable.

(c) $E_\alpha = \{uw,xy\}$, $E_\beta = \{xu,uv,vy\}$. Since $H$ is $(\alpha\beta\gamma)$-rich, the edges $xw,wv,vx$ and $uy$ have type $\gamma$. The type of the remaining edge $wy$ is uniquely determined: the only value of $\varphi\neq\alpha,\beta$ for which the $(\alpha\gamma\varphi)$-triangle and the $(\beta\gamma\varphi)$-triangle are both realizable is $2\beta$. The diagram is drawn in Figure~\ref{fig_cox-pi-4}(i).

(d) $E_\alpha = \{ux,xv,yw\}, E_\beta=\{uw,wv,xy\}$. The only possible type of $uv$ is $2\beta$. Since the $(\gamma,\gamma,2\beta)$-triangle is not realizable, it follows that at most one of the edges $uy,vy$ has type $\gamma$. Since $H$ contains at least two $\gamma$-edges, we can assume that $uy$ and $xw$ have type $\gamma$. The only possible type of $vy$ is then $3\alpha$. See Figure~\ref{fig_cox-pi-4}(ii). We note that the $(\gamma,2\beta,3\alpha)$-triangle is realizable.

(e) $E_\alpha = \{ux,yw\}, E_\beta=\{uw,wv,xy\}$. Since $H$ is $(\alpha\beta\gamma)$-rich, the edges $xw$ and $uy$ have type $\gamma$. Denote the types of $uv, xv, yv$ by $\varphi, \psi, \omega$, respectively. 
Since $\varphi,\psi, \omega \notin \{\alpha,\beta\}$, the triangles of type $(\beta\beta\varphi)$, $(\alpha\varphi\psi)$ and $(\beta\gamma\psi)$ are all realizable only if $\varphi = \gamma$ and $\psi = 2\beta$. The triangles of type $(\beta,2\beta,\omega)$ and $(\alpha\beta\omega)$ are realizable only if $\omega=\gamma$. See Figure~\ref{fig_cox-pi-4}(iii).%

\medskip
Before continuing with $\alpha = \pi/5$, we exclude cases (c) and (e) for both $\alpha = \pi/5$ and $\alpha =2\pi/9$ simultaneously, since the arguments are the same.

(c) $E_\alpha = \{uw,xy\}$, $E_\beta = \{xu,uv,vy\}$. Since the $(\beta\beta\gamma)$-triangle is not realizable, there are at most two $(\alpha\beta\gamma)$-triangles, $uvw$ and $xuw$.

(e) $E_\alpha = \{ux,yw\}, E_\beta=\{uw,wv,xy\}$. Since $H$ is $(\alpha\beta\gamma)$-rich, the edges $xw$ and $uy$ have type $\gamma$. Denote the type of $uv$ by $\varphi$ and the type of $xv$ by $\psi$. Since the $(\beta\beta\varphi)$-triangle is realizable, we have $\varphi \in \{2\alpha,4\alpha\}$. But there is no $\psi\notin \{\alpha,\beta\}$ such that the triangles of types $(\alpha\varphi\psi)$ and $(\beta\gamma\psi)$ are both realizable.


\medskip
2) $\alpha = \pi/5$. Recall that by the spherical law of cosines (Lemma~\ref{lemma_fact}(c)), we have, approximately, $a\sim 0.365$, $b\sim 0.554$ and $c\sim 0.652$.
The spherical triangles of type $(\pi/2,\pi/2,4\alpha)$, $(\pi/2,\pi/2,2\beta)$ and $(\pi/2,3\alpha,2\alpha)$ are not realizable since they have an edge of length $4\pi/5 \sim 2.513$, $2\pi/3 \sim 2.094$ or $\arccos((1-2/\sqrt{5})^{1/2})\sim1.240$, respectively, which cannot be obtained as nonnegative integer combinations of $a,b,c$.
A spherical triangle of type $(\pi/2, 4\alpha, 4\alpha)$ does not exists by Lemma~\ref{lemma_fact}(b). 

\begin{figure}
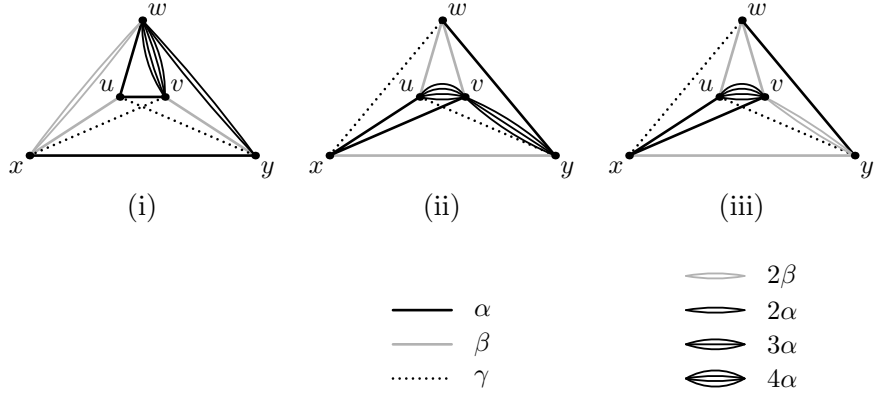

 \begin{center}
  \begin{tabular}{ccc} 
     \includegraphics{obrazky/coxeter.501} &
   \includegraphics{obrazky/coxeter.502} &
    \includegraphics{obrazky/coxeter.503} \\
    (i) & (ii) & (iii) \\
    & & \\
     & \includegraphics{obrazky/popisky.112} & \includegraphics{obrazky/popisky.111} \\
   
  \end{tabular}
  \caption{Coxeter diagrams for $\alpha = \frac{\pi}{5}$.}
\label{fig_cox-pi-5}
 \end{center}
\end{figure}

(a) $E_\alpha = \{uv,xy\}$, $E_\beta = \{xu,yv\}$. Since $H$ is $(\alpha\beta\gamma)$-rich, the diagonals $xv$ and $uy$ are $\gamma$-edges. Let $\varphi, \psi,\omega,\tau$ be the types of $wx,wu,wv,wy$, respectively. Since the $(\alpha\psi\omega)$-triangle and the $(\alpha\varphi\tau)$-triangle have to be realizable and since $\varphi, \psi,\omega,\tau \notin \{\alpha,\beta\}$, for each of the pairs $\{\psi,\omega\}$ and $\{\varphi,\tau\}$ we have the following four options: $\{2\alpha,\gamma\},\{2\alpha,2\beta\},\{\gamma,3\alpha\}$ or $\{\gamma,2\beta\}$. Since the $(\beta\varphi\psi)$-triangle and the $(\beta\omega\tau)$-triangle have to be realizable and since $\varphi, \psi,\omega,\tau \neq 4\alpha$, for each of the pairs $\{\varphi,\psi\}$ and $\{\omega,\tau\}$ we have the following four options: $\{2\alpha,\gamma\},\{2\alpha,3\alpha\},\{\gamma,3\alpha\}$ or $\{3\alpha,2\beta\}$.

Since there are just four $(\alpha\beta\gamma)$-triangles in $H$, the only symmetry of $H$ is the identity, and so $\{\varphi,\psi\} \neq \{\omega,\tau\}.$ Hence, up to symmetry, the $4$-tuple $(\varphi, \psi,\omega,\tau)$ is equal either to $(\pi/2, 2\alpha,\pi/2,3\alpha)$ or $(\pi/2, 2\alpha,2\beta,3\alpha)$. In both cases, the triangle $uyw$ has type $(\pi/2,2\alpha,3\alpha)$, which is not realizable.

(b)  $E_\alpha = \{uv,uw,xy\}$, $E_\beta = \{xu,yv\}$. Like in case (a), the edges $xv$ and $uy$ are of type $\gamma$. Let $\varphi, \psi,\omega$ be the types of $wv,wy,wx$, respectively. Since $\varphi, \psi,\omega \notin \{\alpha,\beta\}$ and the $(\alpha\alpha\varphi)$-triangle must be realizable, we have $\varphi \in \{2\beta,4\alpha\}$. Suppose that $\varphi=2\beta$. Since the 
$(\beta\varphi\psi)$-triangle must be realizable, we have $\psi = 3\alpha$. Since the $(\alpha\psi\omega)$-triangle must be realizable, it follows that $\omega=\pi/2$. But now the triangle $xvw$ is of type $(\pi/2,\pi/2,2\beta)$, which is not realizable; a contradiction. Thus we are left with the option $\varphi = 4\alpha$. Since the $(\alpha\gamma\psi)$-triangle and the $(\beta\varphi\psi)$-triangle must be realizable, we have $\psi=2\alpha$. Since the $(\alpha\psi\omega)$-triangle must be realizable, it follows that $\omega\in \{\pi/2,2\beta\}$.
If $\omega = \pi/2$, the triangle $xvw$ has type $(\pi/2,\pi/2,4\alpha)$, which is not realizable. Therefore, $\omega = 2\beta$. See Figure~\ref{fig_cox-pi-5}(i).

(d) $E_\alpha = \{ux,xv,yw\}, E_\beta=\{uw,wv,xy\}$. The only possible type of $uv$ is $4\alpha$. Since the $(\pi/2,\pi/2,4\alpha)$-triangle is not realizable, it follows that at most one of the edges $uy,vy$ is of type $\gamma$. Since $H$ contains at least two $\gamma$-edges, we can assume that $uy$ and $xw$ are of type $\gamma$. Since the $(\pi/2,4\alpha,4\alpha)$-triangle does not exist, the type of $vy$ is either $3\alpha$ or $2\beta$; see Figure~\ref{fig_cox-pi-5}(ii), (iii).

\medskip 
3)  $\alpha = 2\pi/9$. Recall that by the spherical law of cosines (Lemma~\ref{lemma_fact}(c)), we have, approximately, $a\sim 0.485$, $b\sim 0.680$ and $c\sim 0.812$.
The spherical triangle of type $(\pi/2,\pi/2,2\beta)$ is not realizable since it has an edge of length  $2\pi/3 \sim 2.0944$, which cannot be obtained as a nonnegative integer combination of $a,b,c$.
We show that there is no $H$ satisfying the required conditions.

(a) $E_\alpha = \{uv,xy\}$, $E_\beta = \{xu,yv\}$. 
It is straightforward to check that there are no $\varphi, \psi,\omega, \tau \notin \{\alpha,\beta\}$ such that the triangles of types $(\alpha\psi\omega), (\alpha\varphi\tau), (\beta\varphi\psi), (\beta\omega\tau)$ are all realizable.

(b)  $E_\alpha = \{uv,uw,xy\}$, $E_\beta = \{xu,yv\}$. Since $H$ is $(\alpha\beta\gamma)$-rich, the edges $xv$ and $uy$ are of type $\gamma$. Since there is no realizable $(\beta, 4\alpha, *)$-triangle, the edge $vw$ is of type $2\beta$. Consequently, the edge $yw$ is of type $\alpha + \beta$. Further it follows that $xw$ is a $\gamma$-edge, hence the triangle $xvw$ is of type $(\pi/2,\pi/2,2\beta)$, which is not realizable.

(d) $E_\alpha = \{ux,xv,yw\}, E_\beta=\{uw,wv,xy\}$. There is no $\varphi$ such that the $(\alpha\alpha\varphi)$-triangle and the $(\beta\beta\varphi)$-triangle are both realizable.
\end{proof}

We finish the proof 
using Fiedler's theorem. Compared to Case a), it is now easier to compute the determinants of the matrices corresponding to the Coxeter diagrams from Figures~\ref{fig_cox-pi-4} and~\ref{fig_cox-pi-5}, since we know precise values of all the entries. 

We now list matrices $B_1,B_2, B_3$ corresponding to the diagrams in Figure~\ref{fig_cox-pi-4} and matrices $C_1,C_2,C_3$ corresponding to the diagrams in Figure~\ref{fig_cox-pi-5}. 
Again, the rows and columns of the matrices are indexed by $u,v,w,x,y$, in this order. 

%
%
%
%
%

\[
B_1=
\left(\begin{array}{rrrrr}
-1 & 0.5 & \sqrt{2}/2 & 0.5  & 0 \\
0.5  & -1 & 0 & 0 & 0.5  \\
\sqrt{2}/2 & 0 & -1 & 0 & -0.5  \\
0.5  & 0 & 0 & -1 & \sqrt{2}/2 \\
0 & 0.5  & -0.5  & \sqrt{2}/2 & -1
\end{array}\right)
\]

\[
 B_2=
\left(\begin{array}{rrrrr}
-1 & -0.5 & 0.5 & \sqrt{2}/2 & 0 \\
-0.5 & -1 & 0.5 & \sqrt{2}/2 & -\sqrt{2}/2 \\
0.5 & 0.5 & -1 & 0 & \sqrt{2}/2 \\
\sqrt{2}/2 & \sqrt{2}/2 & 0 & -1& 0.5 \\
0 & -\sqrt{2}/2 & \sqrt{2}/2 & 0.5 & -1
\end{array}\right)
\]

\[ 
B_3 = 
\left(\begin{array}{rrrrr}
-1 & 0 & 0.5 & \sqrt{2}/2 & 0 \\
0 & -1 & 0.5 & -0.5 & 0 \\
0.5 & 0.5 & -1 & 0 & \sqrt{2}/2 \\
\sqrt{2}/2 & -0.5 & 0 & -1 & 0.5 \\
0 & 0 & \sqrt{2}/2 & 0.5 & -1
\end{array}\right)
\]

\[ 
C_1 = 
\left(\begin{array}{rrrrr}
-1 & \frac{1}{4}(\sqrt{5}+1) & \frac{1}{4}(\sqrt{5}+1) & 0.5 & 0 \\
\frac{1}{4}(\sqrt{5}+1) & -1 & -\frac{1}{4}(\sqrt{5}+1) & 0 & 0.5 \\
\frac{1}{4}(\sqrt{5}+1) & -\frac{1}{4}(\sqrt{5}+1) & -1 & -0.5 & \frac{1}{4}(\sqrt{5}-1) \\
0.5 & 0 & -0.5 & -1 & \frac{1}{4}(\sqrt{5}+1) \\
0 & 0.5 & \frac{1}{4}(\sqrt{5}-1) & \frac{1}{4}(\sqrt{5}+1) & -1
\end{array}\right)
\]

\[C_2 = 
\left(\begin{array}{rrrrr}
-1 & -\frac{1}{4}(\sqrt{5}+1) & 0.5 & \frac{1}{4}(\sqrt{5}+1) & 0 \\
-\frac{1}{4}(\sqrt{5}+1) & -1 & 0.5 & \frac{1}{4}(\sqrt{5}+1) & -\frac{1}{4}(\sqrt{5}-1) \\
0.5 & 0.5 & -1 & 0 & \frac{1}{4}(\sqrt{5}+1) \\
\frac{1}{4}(\sqrt{5}+1) & \frac{1}{4}(\sqrt{5}+1) & 0 & -1 & 0.5 \\
0 & -\frac{1}{4}(\sqrt{5}-1) & \frac{1}{4}(\sqrt{5}+1) & 0.5 & -1
\end{array}\right)
\]

\[ 
C_3 = 
\left(\begin{array}{rrrrr}
-1 & -\frac{1}{4}(\sqrt{5}+1) & 0.5 & \frac{1}{4}(\sqrt{5}+1) & 0 \\
-\frac{1}{4}(\sqrt{5}+1) & -1 & 0.5 & \frac{1}{4}(\sqrt{5}+1) & -0.5 \\
0.5 & 0.5 & -1 & 0 & \frac{1}{4}(\sqrt{5}+1) \\
\frac{1}{4}(\sqrt{5}+1) & \frac{1}{4}(\sqrt{5}+1) & 0 & -1 & 0.5 \\
0 & -0.5 & \frac{1}{4}(\sqrt{5}+1) & 0.5 & -1
\end{array}\right)
\]

The matrices have the following determinants, rounded to two decimal places:

\begin{align*}
 \det(B_1)=1/16 \sim 0.06, & & \det(B_2) = 1/8 \sim 0.13, & & \det(B_3) \sim 0.21,    \\
 \det(C_1) \sim 0.16, & & \det(C_2) \sim 0.16, & & \det(C_3) \sim 0.12.  \\
\end{align*}

Since all the determinants are nonzero, simplices corresponding to 
Coxeter diagrams from Figures~\ref{fig_cox-pi-4} and~\ref{fig_cox-pi-5} cannot exist. This finishes the case $\T_0=(\alpha\beta\gamma)$ and also the whole proof of Theorem~\ref{veta_hlavni}.

\section{A few remarks about general dimension $d \ge 5$}
The existence of $k$-reptile simplices in $\R^d$ for $d\ge 5$, $k\ge 3$, $k \neq m^d$, is wide open. Our approach for dimension $4$ does not seem to be powerful enough already in dimension $5$.
Let $S$ be a $k$-reptile simplex in $\R^d$.
Similarly as in Section~\ref{s:main_pf}, we may define an \emph{indivisible} edge-angle of $S$ as a spherical $(d-2)$-simplex 
that cannot be tiled with smaller spherical
$(d-2)$-simplices representing the other edge-angles of $S$ or their mirror images.
Using the same arguments as in the proof of Lemma~\ref{lemma_at-least-4-edges} one can prove a generalized statement saying that
if $\T_0$ is an indivisible edge-angle in $S$, then the edges of $S$ with edge-angle $\T_0$ have at least $D$ different lengths,
where $D$ is the algebraic degree of $k^{-1/d}$ over $\Q$.
In particular, if $d$ is prime then $D=d$.

For $d \ge 5$, we can have up to $(d+1)/2 \ge 3$ indivisible edge-angles in $S$. If $S$ has a nontrivial symmetry, Lemma~\ref{l_pocet_orbit} helps only a little: it still allows up to $(d+1)/2 -1\ge 2$ indivisible edge-angles.

For $d$ prime, it might be beneficial to consider internal angles at inner points of $k$-faces for $k\ge 2$, and use the corresponding generalization of Lemma~\ref{lemma_at-least-4-edges}. However, the combinatorial explosion of possible cases still seems overwhelming. Therefore, we think that to attack the problem in higher dimensions, new ideas and possibly more advanced tools will be needed.

\section*{Acknowledgements}
We used Sage 5.4.1~\cite{sage} and Wolfram{\textbar}Alpha~\cite{wolframalpha} for some technical computations in this paper. We thank Pavel Pat\'ak for his expert assistance with Sage
and for his help with drawing pictures.
We also thank Martin Tancer for careful proofreading of an earlier version and for his valuable suggestions, which were helpful for improving the manuscript.

\bibliographystyle{vlastni}
\bibliography{rep4}

\newpage

\begin{appendix}
\section[]{The program for generating realizable spherical triangles}\label{s:appendix}
\begin{verbatim}
def possible_tilings(u,v,w,lens,min_angle): 
    
    def cos_law(a,b,c):
        z=(cos(c)+cos(a)*cos(b))/(sin(a)*sin(b));
        return numerical_approx(arccos(z));
    
    A=cos_law(w,v,u);
    B=cos_law(w,u,v);
    C=cos_law(u,v,w);
    
    def is_equal(x):
        var("a b c")
        min_error_below=10;meb_i=0;meb_j=0;meb_k=0;
        min_error_above=10;mea_i=0;mea_j=0;mea_k=0;
        for i in range(0,20):
             for j in range(0,20-i):
                for k in range (0,20-i-j):
                    t=x-i*A-j*B-k*C;
                    if (-0.00001<=t<=min_error_below):
                        min_error_below=t;meb_i=i;meb_j=j;meb_k=k;
                    if (-0.00001<=-t<=min_error_above):
                        min_error_above=-t;mea_i=i;mea_j=j;mea_k=k;
        if (min_error_below<0.00001):
            print round(x,5), "& $ x = ",  
                  meb_i*a +meb_j*b + meb_k*c, 
                  "$","!!!\\\\"
            return; 
        print round(x,5), "& $", round(meb_i*A + meb_j*B + meb_k*C,5), 
              "=", meb_i * a  + meb_j * b + meb_k * c, "<", "x", "<", 
              mea_i*a + mea_j *b + mea_k*c, "=", 
              round(mea_i*A + mea_j*B + mea_k*C,5), "$\\\\"
        
    def tries(x,y,z,tlens):
        visited = set();
        def split(n,x,y,z):
            var("alpha beta gamma");
            for i in range(0,x+1):
                for j in range(0,y+1):
                    for k in range (0,z+1):
                        psi=i*u + j*v + k*w
                        phi=(x-i)*u + (y-j)*v + (z-k)*w
                        L = sorted([psi,phi,lens]);
                        if (0 < psi <= phi < pi 
                          and L[1]+L[2] < pi+L[0] 
                          and min_angle<psi 
                          and (psi,phi) not in visited):
                              print n, "&", (x,y,z) ," &";
                              visited.add((psi,phi))
                              print "$(",(i*alpha)+j*beta+k*gamma,",",
                                 (x-i)*alpha+(y-j)*beta+(z-k)*gamma,
                                  ")$&";
                              is_equal(cos_law(psi,phi,lens));
                          
        
        S = (x + y + z - pi);
        for d in range(2,2*tlens/S):
            for k in range(0,(d*S+pi-tlens)/x+1):
                for l in range(0,(d*S+pi-tlens-k*x)/y + 1):
                    for m in range(0,(d*S+pi-tlens-k*x-l*y)/z + 1):
                        if (d*S + pi == k*x + l*y + m*z + tlens):
                            split(d,k,l,m)
    tries(u,v,w,lens)
\end{verbatim}
\end{appendix}

\end{document}